\documentclass[11pt,onecolumn]{IEEEtran}

\usepackage{rrprefsIEEE}
\title{Sample eigenvalue based detection of high-dimensional signals in white noise using relatively few samples}

\begin{document}
%
%
\author{Raj Rao Nadakuditi
        and Alan Edelman
        \thanks{Department of Mathematics, Massachusetts Institute of Technology, Email: \{raj\},\{edelman\}@mit.edu, Phone: (857) 891 8303, Fax: (617) 253-4358 }}

%
%

%
\markboth{Sample eigenvalue based detection}{Rao and Edelman}
%



\maketitle

\begin{abstract}
The detection and estimation of signals in noisy, limited data is a problem of interest to many scientific and engineering communities. We present a mathematically justifiable, computationally simple, sample eigenvalue based procedure for estimating the number of high-dimensional signals in white noise using relatively few samples. The main motivation for considering a sample eigenvalue based scheme is the computational simplicity and the robustness to eigenvector modelling errors which are can adversely impact the performance of estimators that exploit information in the sample eigenvectors. 

There is, however, a price we pay by discarding the information in the sample eigenvectors; we highlight a fundamental asymptotic limit of sample eigenvalue based detection of weak/closely spaced high-dimensional signals from a limited sample size. This motivates our heuristic definition of the \textit{effective number of identifiable signals} which is equal to the number of ``signal'' eigenvalues of the population covariance matrix which exceed the noise variance by a factor strictly greater than \mbox{$1+\sqrt{\frac{\textrm{Dimensionality of the system}}{\textrm{Sample size}}}$}. 

The fundamental asymptotic limit brings into sharp focus why, when there are too few samples available so that the effective number of signals is less than the actual number of signals, underestimation of the model order is unavoidable (in an asymptotic sense) when using any sample eigenvalue based detection scheme, including the one proposed herein. The analysis reveals why adding more sensors can only exacerbate the situation. Numerical simulations are used to demonstrate that the proposed estimator, like Wax and Kailath's MDL based estimator, consistently estimates the true number of signals in the dimension fixed, large sample size limit and the effective number of identifiable signals, unlike Wax and Kailath's MDL based estimator, in the large dimension, (relatively) large sample size limit.
\end{abstract}


\begin{center} \bfseries EDICS Category: SSP-DETC Detection; SAM-SDET Source detection\end{center}
%
\IEEEpeerreviewmaketitle

\section{Introduction}
The observation vector, in many signal processing applications, can be modelled as a superposition of a finite number of signals embedded in additive noise. Detecting the number of signals present becomes a key issue and is often the starting point for the signal parameter estimation problem. When the signals and the noise are assumed, as we do in this paper, to be samples of a stationary, ergodic Gaussian vector process, the sample covariance matrix formed from $m$ observations has the Wishart distribution \cite{wishart28a}. This paper uses an information theoretic approach, inspired by the seminal work of Wax and Kailath \cite{kailath-wax}, for determining the number of signals in white noise from the eigenvalues of the Wishart distributed empirical covariance matrix formed from \textit{relatively few samples}. 

The reliance of the Wax and Kailath estimator and their successors \cite{wong90a,shah94a,liavas01a,fishler00a,fishler02a}, to list a few, on the distributional properties of the eigenvalues of \textit{non-singular} Wishart matrices render them inapplicable in high-dimensional, sample starved settings where the empirical covariance matrix is singular. Ad-hoc modifications to such estimators are often not mathematically justified and it is seldom clear, even using simulations as in \cite{liavas01a}, whether a fundamental limit of detection is being encountered \textit{vis a vis}  the chronically reported symptom of underestimating the number of signals. 

This paper addresses both of these issues using relevant results \cite{johansson_clt_herm,jonsson82a,BaiS04,dumitriu05a} from large random matrix theory. The main contributions of this paper are 1) the development of a mathematically justified, computationally simple, sample eigenvalue based signal detection algorithm that operates effectively in sample starved settings and, 2) the introduction of the concept of \textit{effective number of (identifiable) signals} which brings into sharp focus a fundamental limit in the identifiability, under sample size constraints, of closely spaced/low level signals  using sample eigenvalue based detection techniques of the sort developed in this paper.

The proposed estimator exploits the distributional properties of the trace of powers, \ie, the moments of the eigenvalues, of (singular and non-singular) Wishart distributed large dimensional sample covariance matrices. The definition of \textit{effective number of identifiable signals} is based on the mathematically rigorous results of Baik-Silverstein \cite{BaikS06}, Paul \cite{Paul05a} and Baik et al \cite{BaikBP04} and the heuristic derivation of the first author \cite{raj:thesis}. This concept captures the fundamental limit of sample eigenvalue based detection by  explaining why, in the large system relatively large sample size limit, \textit{if the signal level is below a threshold} that depends on the noise variance, sample size and the dimensionality of the system, \textit{then reliable sample eigenvalue based detection is not possible}. This brings into sharp focus the fundamental undetectability of weak/closely spaced signals using sample eigenvalue based schemes when too few samples are available. Adding more sensors will only exacerbate the problem by raising the detectability threshold.

Conversely, if the signal level is above this threshold, and the dimensionality of the system is large enough, then reliable detection using the proposed estimator is possible. We demonstrate this via numerical simulations that illustrate the superiority of the proposed estimator with the respect to the Wax-Kailath MDL based estimator. Specifically, simulations reveal that while both the new estimator  and the Wax-Kailath MDL estimator are consistent estimators of the number of signals $k$ in the dimensionality $n$ fixed, sample size $m \to \infty $ sense, the MDL estimator is an inconsistent estimator of the \textit{effective number of signals} in the large system, large sample size limit, \ie, in $n,m(n) \to \infty$ limit where the ratio $n/m(n) \to c \in (0,\infty)$ sense. Simulations suggest that the new estimator is a consistent estimator of the effective number of signals in the $n,m(n) \to \infty$ with $n/m(n) \to c\in (0,\infty)$ sense. We note that simulations will demonstrate the applicability of the proposed estimator in moderate dimensional settings as well.

The paper is organized as follows. The problem formulation in Section \ref{sec:lrcf problem formulation} is followed by a summary in Section \ref{sec:relevant rmt} of the relevant properties of the eigenvalues of large dimensional Wishart distributed sample covariance matrices.  An estimator for the number of signals present that exploits these results is derived in Section \ref{sec:number of signals}. An extension of these results to the frequency domain is discussed in Section \ref{sec:lrcf frequency domain}. Consistency of the proposed estimator and the concept of \textit{effective number of signals} is discussed in Section \ref{sec:consistent number of signals}. Simulation results that illustrate the superior performance of the new method in high dimensional, sample starved settings are presented in Section \ref{sec:lrcf simulations}; some concluding remarks are presented in Section \ref{sec:lrcf conclusion}.

\section{Problem formulation}
\label{sec:lrcf problem formulation}
We observe $m$ samples (``snapshots'') of possibly signal bearing $n$-dimensional snapshot vectors ${\bf x}_{1}, \ldots, {\bf x}_{m}$ where for each $i$, ${\bf x}_{i} \sim \mathcal{N}_{n}(0,{\bf R})$ and ${\bf x}_{i}$ are mutually independent.  The snapshot vectors are modelled as 
\begin{equation}\label{eq:superposition problem}
{\bf x}_{i} = 
\begin{cases}
\,{\bf z}_{i} &\textrm{No Signal} \\
{\bf A}\,{\bf s}_{i}+{\bf z}_{i} &\textrm{Signal Present}\\
\end{cases} \qquad \textrm{for } i = 1,\ldots,m,
\end{equation}
where ${\bf z}_{i} \sim \mathcal{N}_{n}(0,\sigma^{2}{\bf I})$, denotes an $n$-dimensional (real or circularly symmetric complex) Gaussian noise vector where $\sigma^{2}$ is assumed to be unknown, ${\bf s}_{i} \sim \mathcal{N}_{k}({\bf 0},{\bf R}_{s})$ denotes a $k$-dimensional (real or circularly symmetric complex) Gaussian signal vector with covariance ${\bf R}_{s}$, and ${\bf A}$ is a $n \times k$  unknown non-random matrix. In array processing applications, the $j$-th column of the matrix ${\bf A}$ encodes the parameter vector associated with the $j$-th signal whose magnitude is described by the $j$-the element of ${\bf s}_{i}$.

Since the signal and noise vectors are independent of each other, the covariance matrix of ${\bf x}_{i}$ can be decomposed as
\begin{equation}
{\bf R} =  {\bf \Psi} + \sigma^{2} {\bf I}
\end{equation}
where
\begin{equation}
{\bf \Psi} = {\bf A}{\bf R}_{s}{\bf A}' ,
\end{equation}
with $'$ denoting the conjugate transpose. Assuming that the matrix ${\bf A}$ is of full column rank, \ie, the columns of ${\bf A}$ are linearly independent, and that the covariance matrix of the signals ${\bf R}_{s}$ is nonsingular, it follows that the rank of ${\bf \Psi}$ is $k$. Equivalently, the $n-k$ smallest eigenvalues of ${\bf \Psi}$ are equal to zero. 

If we denote the eigenvalues of ${\bf R}$ by $\lambda_{1}\geq \lambda_{2} \geq \ldots \geq \lambda_{n}$ then it follows that the smallest $n-k$ eigenvalues of ${\bf R}$ are all equal to $\sigma^{2}$ so that
\begin{equation}
\lambda_{k+1}= \lambda_{k+2} = \ldots = \lambda_{n} = \lambda = \sigma^{2}.
\end{equation}
Thus, if the true covariance matrix ${\bf R}$ were known \textit{apriori}, the dimension of the signal vector $k$ can be determined from the multiplicity of the smallest eigenvalue of ${\bf R}$. When there is no signal present, all the eigenvalues of ${\bf R}$ will be identical. The problem in practice is that the covariance matrix ${\bf R}$ is unknown so that such a straight-forward algorithm cannot be used. The signal detection and estimation problem is hence posed in terms of an inference problem on $m$ samples of $n$-dimensional multivariate real or complex Gaussian snapshot vectors.  

Inferring the number of signals from these $m$ samples reduces the signal detection problem to a model selection problem for which there are many approaches. A classical approach to this problem, developed by Bartlett \cite{bartlett54a} and Lawley \cite{lawley56a}, uses a sequence of hypothesis tests. Though this approach is sophisticated, the main problem is the subjective judgement needed by the practitioner in selecting the threshold levels for the different tests. 

Information theoretic criteria for model selection such as those developed by Akaike \cite{akaike73a,akaike74a}, Schwartz \cite{schwartz78a} and Rissanen \cite{rissanen78a} address this problem by proposing the selection of the model which gives the minimum \textit{information criteria}. The criteria for the various approaches is generically a function of the log-likelihood of the maximum likelihood estimator of the parameters of the model and a term which depends on the number of parameters of the model that penalizes overfitting of the model order. 

For the problem formulated above, Kailath and Wax \cite{kailath-wax} propose an estimator for the number of signals (\textit{assuming $m>n$ and ${\bf x}_{i} \in \mathbb{C}^{n}$}) based on the eigenvalues $l_{1} \geq l_{2} \geq \ldots \geq l_{n}$ of the sample covariance matrix (SCM) defined by
\begin{equation}
\widehat{{\bf R}} = \frac{1}{m} \sum_{i=1}^{m} {\bf x}_{i} {\bf x}_{i}' = \frac{1}{m} {\bf X}{\bf X}'
\end{equation}
where ${\bf X}= [{\bf x}_{1}| \ldots| {\bf x}_{m}]$ is the matrix of observations (samples).  The Akaike Information Criteria (AIC) form of the estimator is given by
\begin{equation}\label{eq:aic est}
\hat{k}_{{\rm AIC}} = \argmin  -2(n-k)m \log \frac{g(k)}{a(k)}+ 2k(2n-k) \qquad \textrm{for } k \in \mathbb{N}: 0 \leq k < n
\end{equation}
while the Minimum Descriptive Length (MDL) criterion is given by
\begin{equation}\label{eq:mdl est}
\hat{k}_{{\rm MDL}} = \argmin - (n-k)m\log \frac{g(k)}{a(k)} + \frac{1}{2}k(2n-k)\log m \qquad \textrm{for } k \in \mathbb{N}: 0 \leq k < n
\end{equation}
where $g(k) = \prod_{j=k+1}^{n} l_{j}^{1/(n-k)}$ is the geometric mean of the $n-k$ smallest sample eigenvalues and $a(k)= \frac{1}{n-k} \sum_{j = k+1}^{n} l_{j}$ is their arithmetic mean. 

It is known \cite{kailath-wax} that the AIC form inconsistently estimates the number of signals, while the MDL form estimates the number of signals consistently in the classical $n$ fixed, $m \to \infty$ sense. The simplicity of the estimator, and the large sample consistency are among the primary reasons why the Kailath-Wax MDL estimator continues to be employed in practice \cite{vantrees02a}. In the two decades since the publication of the WK paper, researchers have come up with many innovative solutions for making the estimators more ``robust'' in the sense that estimators continue to work well in settings where the underlying assumptions of snapshot and noise Gaussianity and inter-snapshot independence (e.g. in the presence of multipath) can be relaxed as in the work of Zhao et al \cite{zhao86a,zhao86b}, Xu et al \cite{wong90a}, and Stoica-Cedervall \cite{stoica97a}among others \cite{fishler05a}. 

Despite its obvious practical importance, the robustness of the algorithm to model mismatch is an issue that we shall not address in this paper. Instead we aim to revisit the original problem considered by Wax and Kailath with the objective of designing an estimator that is robust to high-dimensionality and sample size constraints. We are motivated by the observation that the most important deficiency of the Wax-Kailath estimator and its successors that is yet to be satisfactorily resolved occurs in the setting where the sample size is smaller  than the number of sensors, \ie, when $m < n$, which  is increasingly the case in many state-of-the-art radar and sonar systems where the number of sensors exceeds the sample size by a factor of $10-100$ \cite{abb:private}. In this situation, the SCM is singular and the estimators  become degenerate, as seen in (\ref{eq:aic est}) and (\ref{eq:mdl est}). Practitioners often overcome this in an ad-hoc fashion by, for example, restricting $k$ in (\ref{eq:mdl est}) to integer values in the range $0 \leq k < \min(n,m)$ so that
\begin{equation}\label{eq:WK MDL mod}
\hat{k}_{\overline{{\rm MDL}}} = \argmin - (n-k)m\log \frac{g(k)}{a(k)} + \frac{1}{2}k(2n-k)\log m \qquad \textrm{for } k \in \mathbb{N}: 0 \leq k < \min(n,m)
\end{equation}
Since large sample, \ie, $m \gg n$, asymptotics \cite{Anderson63} were used to derive the estimators in \cite{kailath-wax}, there is no rigorous theoretical justification for such a reformulation even if the simulation results suggest that the WK estimators are working ``well enough.'' 

This is true for other sample eigenvalue based solutions found in the literature that exploit the sample eigenvalue order statistics \cite{fishler00a,fishler02a}, employ a Bayesian framework by imposing priors on the number of signals \cite{bansal91a}, involve solving a set of possibly high-dimensional non-linear equations \cite{wong90a},  or propose sequential hypothesis testing procedures \cite{shah94a}. The fact that these solutions are computationally more intensive  or require the practitioner to set subjective threshold levels makes them less attractive than the WK MDL solution; more importantly,  they do not address the sample starved setting in their analysis or their simulations either. 

For example, in \cite{fishler02a}, Fishler et al use simulations to illustrate the performance of their algorithm with $n=7$ sensors and a sample size $m>500$ whereas in a recent paper \cite{fishler05a}, Fishler and Poor illustrate their performance with $n=10$ sensors and $m>50$ samples. Van Trees  discusses the various techniques for estimating the number of signals in Section 7.8 of \cite{vantrees02a}; the sample starved setting where $m=O(n)$ or $m < n$ is not treated in the simulations either. 

There is however, some notable work on detecting the number of signals using short data records. Particle filter based techniques \cite{larocque02a}, have proven to be particularly useful in such short data record settings. Their disadvantage, from our perspective, is that they require the practitioner to the model the eigenvectors of the underlying population covariance matrix as well; this makes them especially sensitive to model mismatch errors that are endemic to high-dimensional settings. 

This motivates our development of a mathematically justifiable, sample eigenvalue based estimator with a computational complexity comparable to that of the modified WK estimator in (\ref{eq:WK MDL mod}) that remains robust to high-dimensionality and sample size constraints. The proposed new estimator given by:

\theorembox{
\begin{subequations}\label{eq:new estimator}
\begin{align}
t_{k} &= \left[(n-k)\dfrac{\sum_{i=k+1}^{n} l_{i}^{2}}{(\sum_{i=k+1}^{n} l_{i})^{2}}- \left(1+\dfrac{n}{m}\right) \right]n-\left(\dfrac{2}{\beta}-1\right)\dfrac{n}{m}\\[0.1in]
\hat{k}_{{\rm NEW}} &= \argmin_{k} \,\left\{\dfrac{\beta}{4} \left[\dfrac{m}{n}\right]^{2}\, t_{k}^{2} \right\} + 2(k+1) \qquad \textrm{for } k \in \mathbb{N}: 0\leq k < \min(n,m).\ 
\end{align}
\end{subequations}
Here $\beta = 1$ if ${\bf x}_{i} \in \mathbb{R}^{n}$, and $\beta = 2$ if ${\bf x}_{i} \in \mathbb{C}^{n}$.
}
In (\ref{eq:new estimator}), the $l_{i}$'s are the eigenvalues of the sample covariance matrix $\widehat{{\bf R}}$. An implicit assumption in the derivation of (\ref{eq:new estimator}) is that the number of signals is much smaller than the system size, \ie, $k \ll n$. A rather important consequence of our sample eigenvalue based detection scheme, as we shall elaborate in Section \ref{sec:consistent number of signals}, is that it just might not be possible to detect low level/closely spaced signals when there are too few samples available. 

To illustrate this effect, consider the situation where there are two uncorrelated (hence, independent) signals so that ${\bf R}_{s} = \textrm{diag}(\sigma_{{\rm S}1}^{2},\sigma_{{\rm S}2}^{2})$. In (\ref{eq:superposition problem}) let ${\bf A} = [{\bf v}_{1} {\bf v}_{2}]$, as in a sensor array processing application so that ${\bf v}_{1} \equiv {\bf v}(\theta_{1})$ and ${\bf v}_{2} \equiv {\bf v}_{2}(\theta_{2})$ encode the array manifold vectors for a source and an interferer with powers $\sigma_{{\rm S}1}^{2}$ and $\sigma_{{\rm S}2}^{2}$, located at $\theta_{1}$ and $\theta_{2}$, respectively. The covariance matrix is then given by 
\begin{equation}
{\bf R} = \sigma_{{\rm S}1}^{2} {\bf v}_{1}{\bf v}_{1}'+ \sigma_{{\rm S}2}^{2}  {\bf v}_{2}{\bf v}_{2}' + \sigma^{2} {\bf I}.
\end{equation}
In the special situation when  $\parallel\! {\bf v}_{1} \!\parallel = \parallel\! {\bf v}_{2}\! \parallel = \parallel\! {\bf v}\! \parallel$ and $\sigma_{{\rm S1}}^{2} = \sigma_{{\rm S2}}^{2} = \sigma_{{\rm S}}^{2}$, we can (in an asymptotic sense) reliably detect the presence of \textit{both signals} from the sample eigenvalues alone whenever
\begin{empheq}[
box=\setlength{\fboxrule}{1pt}\fbox]{equation}
\label{eq:array tradeoff intro}
\textrm{Asymptotic identifiability condition}: \qquad \sigma_{{\rm S}}^{2} \parallel \! {\bf v} \!\parallel^{2} \left(1-\dfrac{|\langle {\bf v}_{1},{\bf v}_{2}\rangle |}{\parallel \! {\bf v}\parallel}\right)  > \sigma^{2} \sqrt{\dfrac{n}{m}}  
\end{empheq}
If the signals are not strong enough or not spaced far enough part, then not only will proposed estimator consistently underestimate the number of signals but so will any other sample eigenvalue based detector.

The concept of the effective number of signals provides insight into the fundamental limit, due to snapshot constraints in high-dimensional settings, of reliable signal detection by eigen-inference , \ie, by using the sample eigenvalues alone. This helps identify scenarios where algorithms that exploit any structure in the eigenvectors of the signals, such as the MUSIC and the Capon-MVDR algorithms in sensor array processing \cite{vantrees02a} or particle filter based techniques \cite{larocque02a}, might be better able to tease out lower level signals from the background noise. It is worth noting that the proposed approach remain relevant in situations where the eigenvector structure has been identified. This is because eigen-inference methodologies are inherently robust to eigenvector modelling errors that occur in high-dimensional settings. Thus the practitioner may use the proposed methodologies to complement and ``robustify'' the inference provided by  algorithms that exploit the eigenvector structure.

\section{Pertinent results from random matrix theory}\label{sec:relevant rmt}

Analytically characterizing the distribution of the sample eigenvalues, as a function of the population eigenvalues, is the first step in designing a sample eigenvalue based estimator  that is robust to high-dimensionality and sample size constraints. For arbitrary covariance ${\bf R}$,  the joint density function of the eigenvalues $l_{1},\ldots,l_{n}$ of the SCM $\widehat{\bf R}$ \underline{when $m>n+1$} is shown to be given by \cite{Anderson63}
\begin{equation}\label{eq:gmm joint density arb covariance}
f(l_{1},\ldots,l_{n}) = \widetilde{Z}_{n,m}^{\beta} \sum_{i=1}^{n} l_{i}^{\beta(m-n+1)/2-1} \prod_{i<j}^{n} |l_{i}-l_{j}|^{\beta} \prod_{i=1}^{n} dl_{i} \int_{{\bf Q}} \exp\left(-\dfrac{m\beta}{2}\, \Tr\,\left({\bf R}^{-1}{\bf Q}\widehat{\bf R}{\bf Q}'\right)\right) d{\bf Q}
\end{equation}
where $l_{1}>\ldots>l_{n}>0$, $\widetilde{Z}_{n,m}^{\beta}$ is a normalization constant, and $\beta=1$ (or $2$) when $\widehat{\bf R}$ is real (resp. complex). In (\ref{eq:gmm joint density arb covariance}), ${\bf Q} \in {\bf O}(n)$ when $\beta = 1$ while ${\bf Q} \in {\bf U}(n)$ when $\beta=2$ where ${\bf O}(n)$ and ${\bf U}(n)$ are, respectively, the set of $n \times n$ orthogonal and unitary matrices with Haar measure. 

Note that the exact characterization of the joint density of the eigenvalues in (\ref{eq:gmm joint density arb covariance}) involves  a multidimensional integral over the orthogonal (or unitary) group. This makes it intractable for analysis without resorting to asymptotics. Anderson's landmark paper \cite{Anderson63} does just that by characterizing the distribution of the sample eigenvalues using large sample asymptotics. A classical result due to Anderson establishes the consistency of the sample eigenvalues in the dimensionality $n$ fixed, sample size $m \to \infty$ asymptotic regime \cite{Anderson63}. When the dimensionality is small and there are plenty of samples available, Anderson's analysis suggests that the sample eigenvalues will be (roughly) symmetrically centered around the population eigenvalues. When the dimensionality is large, and the sample size is relatively small, Anderson's prediction of sample eigenvalue consistency is in stark contrast to the asymmetric spreading of the sample eigenvalues that is observed in numerical simulations. This is illustrated in Figure \ref{fig:illustration of spiked blurring} where the $n = 20$ eigenvalues of a SCM formed from $m=20$ samples are compared with the eigenvalues of the underlying population covariance matrix.

\begin{figure}[t]
\subfigure[True Eigenvalues. ]{
\includegraphics[width=3.2in]{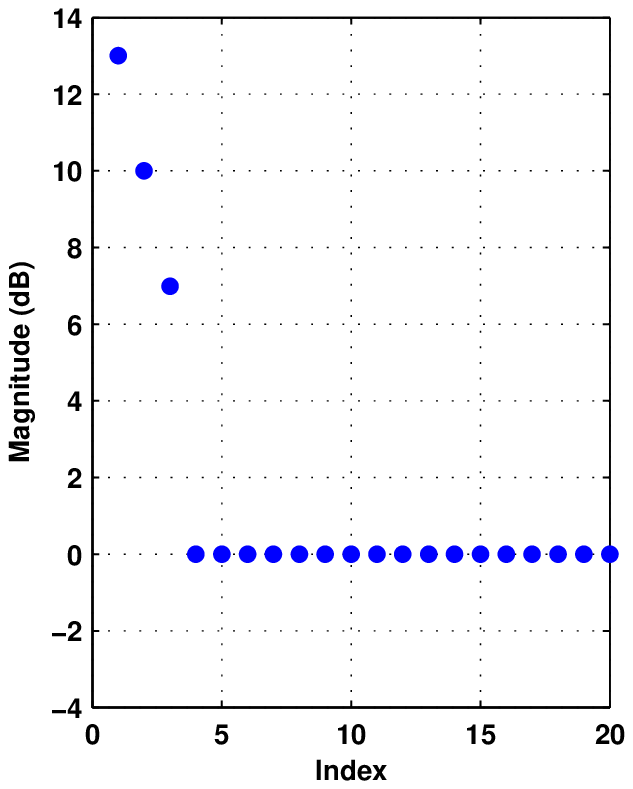}
}
\subfigure[Sample eigenvalues formed from  $20$ snapshots.]{
\includegraphics[width=3.2in]{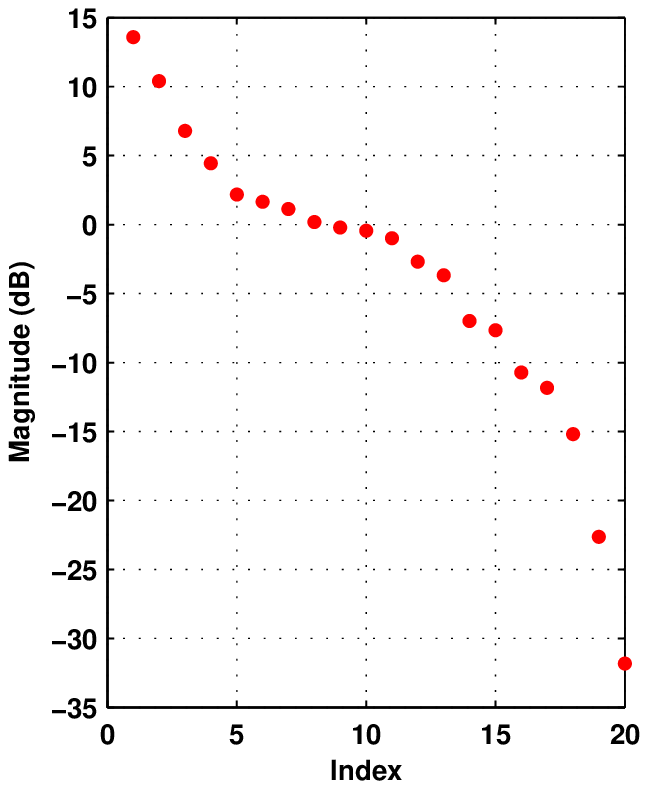}
}
\caption{Blurring of sample eigenvalues relative to the population eigenvalues when are there a finite number of snapshots.}
\label{fig:illustration of spiked blurring}
\end{figure}

The role of random matrix theory comes in because of new analytical results that are able to precisely describe the spreading of the sample eigenvalues exhibited in Figure \ref{fig:illustration of spiked blurring}. Since our new estimator explicitly exploits these analytical results, the use of our estimator in high-dimensional, sample starved settings is more mathematically justified than other sample eigenvalue based approaches found in the literature  that explicitly use Anderson's sample eigenvalue consistency results. See, for example \cite[Eq. (13a), pp. 389]{kailath-wax}, \cite[Eq (5)., pp. 2243]{fishler00a}.

We argue that this is particularly so in settings where $m<n$, where practitioners (though, not the original authors!) have often invoked the equality between the non-zero eigenvalues of the $\widehat{{\bf R}} = (1/m) {\bf X}{\bf X}'$ and the matrix $(1/m){\bf X}'{\bf X}$ to justify ad-hoc modifications to estimators that use only the non-zero eigenvalues of the SCM, as in (\ref{eq:WK MDL mod}). We contend that such an ad-hoc modification to any estimator that explicitly uses Anderson's sample eigenvalue consistency results  is mathematically unjustified because the sample eigen-spectrum blurring, which is only exacerbated in the $m<n$ regime, remains unaccounted for. 

Before summarizing the pertinent results, we note that the analytical breakthrough is a consequence of considering the large system size, relatively large sample size asymptotic regime as opposed to ``classical'' fixed system size, large sample size asymptotic regime. Mathematically speaking, the new results describe the distribution of the eigenvalues in $n,m \to \infty$ with $n/m \to c \in (0,\infty)$ asymptotic regime as opposed to the $n$ fixed, $m \to \infty$ regime \textit{\`{a} la} Anderson. We direct the reader to Johnstone's excellent survey for a discussion on these asymptotic regimes \cite[pp. 9]{imj06a} and much more.

\subsection{Eigenvalues of the signal-free SCM}
A central object in the study of large random matrices is the empirical distribution function (e.d.f.) of the eigenvalues, which for an arbitrary matrix ${\bf A}$ with $n$ real eigenvalues (counted with multiplicity), is defined as 
\begin{equation}\label{eq:edf definition}
F^{{\bf A}}(x) = \frac{\textrm{Number of eigenvalues of } {\bf A} \leq x}{n}. 
\end{equation} 
For a broad class of random matrices, the sequence of e.d.f.'s can be shown to converge in the $n \to \infty$ limit to a non-random distribution function \cite{raj:acta}. Of particular interest is the convergence of the e.d.f. of the signal-free SCM which is described next.

\begin{proposition}\label{prop:mandp density}
Let $\widehat{{\bf R}}$ denote a signal-free sample covariance matrix formed from an $n \times m$ matrix of observations with i.i.d. Gaussian samples of mean zero and variance $\lambda=\sigma^{2}$.  Then the e.d.f. $F^{{\widehat{\bf R}}}(x) \to F^{W}(x)$ almost surely for every $x$, as $m,n \to \infty$ and $c_{m} = n/m \to c$ where 
\begin{equation}\label{eq:mandp density}
dF^{W}(x) = \max\left(0,\left(1-\frac{1}{c}\right)\right)\delta(x)+\frac{\sqrt{(x-a_{-})(a_{+}-x)}}{2\pi\lambda x c} \mathbb{I}_{[a_{-},a_{+}]}(x)\, dx,
\end{equation} 
with $a_{\pm}=\lambda(1\pm\sqrt{c})^{2}$, $\mathbb{I}_{[a,b]}(x) = 1$ when $a \leq x \leq b$ and zero otherwise, and $\delta(x)$ is the Dirac delta function. 
\end{proposition}
\begin{proof}
This result was proved in \cite{marcenko67a,wachter78a} in very general settings. Other proofs include \cite{jonsson82a,silverstein95a,silverstein95b}. The probability density in (\ref{eq:mandp density}) is often referred to as the Mar\v{c}enko-Pastur density.
\end{proof}

Figure \ref{fig:wishart mandp} plots the Mar\v{c}enko-Pastur density in (\ref{eq:mandp density}) for $\lambda = 1$ and different values of $c = \lim n/m$. Note that as $c \to 0$, the eigenvalues are increasingly clustered around $x = 1$ but for modest values of $c$, the spreading is quite significant.

\begin{figure}
\centering
\includegraphics[width=5.1in]{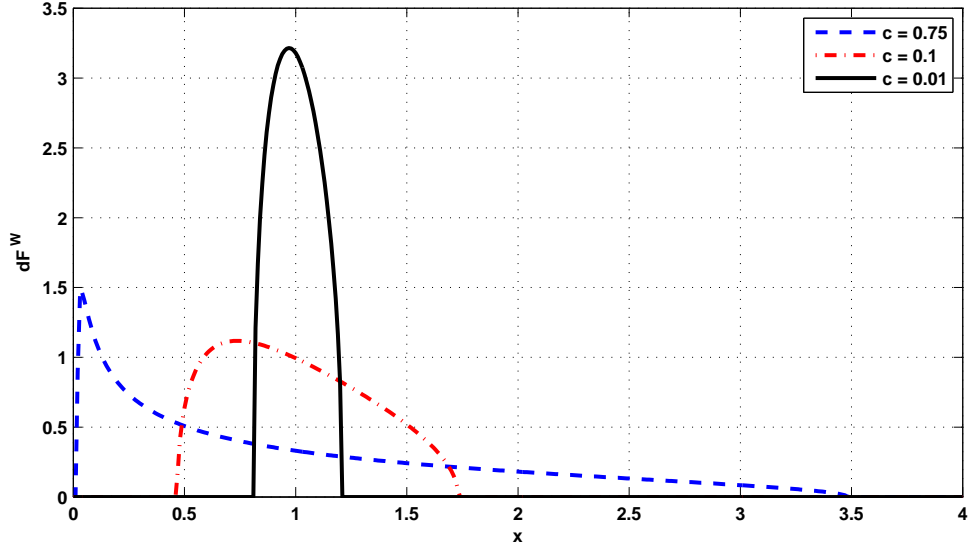}
\caption{The Mar\v{c}enko-Pastur density, given (\ref{eq:mandp density}), for the eigenvalues of the signal-free sample covariance matrix with noise variance $1$ and   $c = \lim n/m$.}
\label{fig:wishart mandp} 
\end{figure}

The almost sure convergence of the e.d.f. of the signal-free SCM implies that the moments of the eigenvalues converge almost surely, so that
\begin{equation}\label{eq:convergence of wishart moments}
\frac{1}{n} \sum_{i=1}^{n} l_{i}^{k} \overset{a.s.}{\longrightarrow} \int x^{k} dF^{W}(x) =: M^{W}_{k}.
\end{equation}
The moments of the Mar\v{c}enko-Pastur density are given by \cite{jonsson82a,dumitriu03a}
\begin{equation}
M^{W}_{k} = \lambda^{k} \sum_{j=0}^{k-1} c^{j} \frac{1}{j+1} {k \choose j} {k-1 \choose j}.
\end{equation}
For finite $n$ and $m$, the sample moments, \ie, $\frac{1}{n} \sum_{i=1}^{n} l_{i}^{k}$ will fluctuate about these limiting values. The precise nature of the fluctuations is described next.

\begin{proposition}\label{prop:mandp fluctuations}
If $\widehat{{\bf R}}$ satisfies the hypotheses of Proposition \ref{prop:mandp density} for some $\lambda$ then as $m,n \to \infty$ and $c_{m} = n/m \to c \in (0,\infty)$, then 
\begin{equation}\label{eq:subs values}
\smallbox{
n\,\left(\begin{bmatrix}
\frac{1}{n}\sum_{i=1}^{n} l_{i}  \\
\\
\frac{1}{n}\sum_{i=1}^{n} l_{i}^{2}  
\end{bmatrix}-
\begin{bmatrix}
\lambda \\
\\
\lambda^{2}(1+c)
\end{bmatrix}
\right)
\overset{\mathcal{D}}{\longrightarrow} \mathcal{N}\left(\underbrace{\begin{bmatrix} 0 \\(\frac{2}{\beta}-1)\lambda^{2}c\end{bmatrix}}_{={\bm{\mu}_{Q}}},\frac{2}{\beta}\underbrace{\begin{bmatrix}
\lambda^{2}c & 2\lambda^{3}c\,(c+1) \\
2\lambda^{3}c\,(c+1)  & 2\lambda^{4}c\,(2c^2+5c+2)\\
\end{bmatrix}}_{= {\bf Q}}\right)
}
\end{equation}
where the convergence is in distribution.
\end{proposition}
\begin{proof}
This result appears in \cite{johansson_clt_herm,jonsson82a} for the real case and in \cite{BaiS04} for the real and complex cases. The result for general $\beta$ appears in Dumitriu and Edelman \cite{dumitriu05a}. 
\end{proof}

We now use the result in Proposition \ref{prop:mandp fluctuations} to develop a test statistic $q_{n}$ whose distribution is independent of the unknown noise variance $\lambda$. The distributional properties of this test statistic are described next.

\begin{proposition}\label{prop:main prop}
Assume $\widehat{{\bf R}}$ satisfies the hypotheses of Proposition \ref{prop:mandp density} for some $\lambda$. Consider the statistic
\[
q_{n} = \dfrac{\frac{1}{n}\sum_{i=1}^{n} l_{i}^2}{\left(\frac{1}{n}\sum_{i=1}^{n} l_{i}\right)^{2}}.
\]
Then as $m,n \to \infty$ and $c_{m} = n/m \to c \in (0,\infty)$, 
\begin{equation}\label{eq:main dist results}
n\left[q_{n}-(1+c)\right] \overset{\mathcal{D}}{\longrightarrow} \mathcal{N}\left(\left(\dfrac{2}{\beta}-1\right)c,\dfrac{4}{\beta}c^{2}\right) 
\end{equation}
where the convergence is in distribution.
\end{proposition}
\begin{proof}
Define the function $g(x,y) = y/x^2$. Its gradient vector $\bm{\nabla}\!g(x,y)$ is given by
\begin{equation}\label{eq:delta method gradient}
\bm{\nabla}\!g(x,y) := [\partial_{x} g(x,y) \,\,\,\:\:~~\partial_{y} g(x,y)]^{T} = [-2y/x^{3} \,\,\,\:\:~~~1/x^{2}]^{T}.
\end{equation}
The statistic $q_{n}$ can be written in terms of $g(x,y)$ as simply $q_{n}=g(\frac{1}{n}\sum_{i=1}^{n} l_{i},\frac{1}{n}\sum_{i=1}^{n} l_{i}^{2})$. The limiting distribution of $q_{n}$ can be deduced from the distributional properties of $\frac{1}{n}\sum_{i=1}^{n} l_{i}$ and $\frac{1}{n}\sum_{i=1}^{n} l_{i}^{2}$ established in Proposition \ref{prop:mandp fluctuations}. Specifically, by an application of the delta method \cite{casella90a}, we obtain that as $n,m \to\infty$ with $n/m \to c \in (0,\infty)$, 
\[
n\left[q_{n}-g(\lambda,\lambda^{2}(1+c))\right] \overset{\mathcal{D}}{\longrightarrow} \mathcal{N}(\mu_{q},\sigma^{2}_{q})
\] 
where the mean $\mu_{q}$ and the variance $\sigma^{2}_{q}$ are given by
\begin{subequations}
\begin{align}
\mu_{q} &= \bm{\mu}_{Q}^{T}  \bm{\nabla}\!g\left(\lambda,\lambda^{2}(1+c)\right),\label{eq:delta method mean}\\ 
\sigma^{2}_{q} &= \dfrac{2}{\beta}\bm{\nabla}\! g\left(\lambda,\lambda^{2}(1+c)\right)^{T}{\bf Q}\bm{\nabla}\!g\left(\lambda,\lambda^{2}(1+c)\right)
 \label{eq:delta method variance}
\end{align}
\end{subequations}
Substituting, the expressions for $\bm{\mu}_{Q}$ and ${\bf Q}$ given in (\ref{eq:subs values}), in (\ref{eq:delta method mean}) and (\ref{eq:delta method variance}) gives us the required expressions for the mean and the variance of the normal distribution on the right hand side of (\ref{eq:main dist results})
\end{proof}

\subsection{Eigenvalues of the signal bearing SCM}
When there are $k$ signals present then, in the $n \to \infty$ limit, where $k$ is kept fixed, the \textit{limiting} e.d.f. of $\widehat{{\bf R}}$ will still be given by Proposition \ref{prop:mandp density}. This is because the e.d.f., defined as in (\ref{eq:edf definition}),  weights the contribution of every eigenvalue equally so that effect of the $k/n$ fraction of ``signal'' eigenvalues vanishes in the $n\to \infty$ limit. 

Note, however, that in the signal-free case, \ie,  when $k=0$, Proposition \ref{prop:mandp density} and the result in \cite{yinbai88a} establish the almost sure convergence of the largest eigenvalue of the SCM to $\lambda(1+\sqrt{c})^{2}$. In the signal bearing case, a so-called phase transition phenomenon is observed, in that the largest eigenvalue will converge to a limit different from that in the signal-free case only if the ``signal'' eigenvalues are above a certain threshold. This is described next.

\begin{proposition}\label{prop:spiked convergence}
Let $\widehat{{\bf R}}$ denote a sample covariance matrix formed from an $n \times m$ matrix of Gaussian observations whose columns are independent of each other and identically distributed with mean ${\bf 0}$ and covariance ${\bf R}$. Denote the eigenvalues of ${\bf R}$ by $\lambda_{1} \geq \lambda_{2} > \ldots \geq \lambda_{k} > \lambda_{k+1} = \ldots \lambda_{n} = \lambda$. Let $l_{j}$ denote the $j$-th largest eigenvalue of $\widehat{{\bf R}}$. Then as $n,m \to \infty$ with $c_{m} = n/m \to c \in (0,\infty)$,
\begin{equation}
l_{j} \to 
\begin{cases}
\lambda_{j} \left( 1+ \dfrac{\lambda\,c}{\lambda_{j}-\lambda}\right) & {\rm if } \qquad \lambda_{j} > \lambda\,(1+\sqrt{c})\\
& \\
\lambda\,(1+\sqrt{c})^{2} & {\rm if } \qquad \lambda_{j} \leq \lambda(1+\sqrt{c})\\
\end{cases}
\end{equation}
for $j =1, \ldots, k$ and the convergence is almost surely.
\end{proposition}
\begin{proof}
This result appears in \cite{BaikS06} for very general settings. A matrix theoretic proof for the real valued SCM case may be found in \cite{Paul05a} while a determinental proof for the complex case may be found in \cite{BaikBP04}. A heuristic derivation that relies on an interacting particle system interpretation of the sample eigenvalues appears in \cite{raj:thesis}.
\end{proof}

For ``signal'' eigenvalues above the threshold described in Proposition \ref{prop:spiked convergence}, the fluctuations about the asymptotic limit are described next.

\begin{proposition}\label{prop:spiked fluctuations}
Assume that $\widehat{{\bf R}}$ and ${\bf R}$ satisfy the hypotheses of Proposition \ref{prop:spiked convergence}. If $\lambda_{j}> \lambda(1+\sqrt{c})$ has multiplicity $1$ and  if $\sqrt{m}|c-n/m| \to 0$ then 
\begin{equation}\label{eq:signal fluctuations variance}
\sqrt{n}\left[l_{j} - \lambda_{j} \left( 1+ \dfrac{\lambda\,c}{\lambda_{j}-\lambda}\right)\right] \overset{\mathcal{D}}{\longrightarrow} \mathcal{N}\left(0,\dfrac{2}{\beta}\lambda_{j}^{2} \left(1- \dfrac{c}{(\lambda_{j}-\lambda)^{2}} \right)\right)
\end{equation}
where the convergence in distribution is almost surely.
\end{proposition}
\begin{proof}
A matrix theoretic proof for the real case may be found in \cite{Paul05a} while a determinental proof for the complex case may be found in \cite{BaikBP04}. The result has been strengthened for non-Gaussian situations by Baik and Silverstein for general $c \in (0,\infty)$ \cite{silverstein:private}. 
\end{proof}

\section{Estimating the number of signals}\label{sec:number of signals}
 We derive an information theoretic estimator for the number of signals by exploiting the distributional properties of the moments of eigenvalues of the (signal-free) SCM given by Propositions \ref{prop:mandp fluctuations} and \ref{prop:main prop}, as follows. The overarching principle used is that, given an observation ${\bf y} = [y(1), \ldots, y(N)]$ and a family of models, or equivalently a parameterized family of probability densities $f({\bf y}|\bm{\theta})$ indexed by the parameter vector $\bm{\theta}$, we select the model which gives the minimum Akaike Information Criterion (AIC) \cite{akaike74a} defined by
\begin{equation}\label{eq:aic criterion}
{\rm AIC}_{k} = -2 \log f({\bf y} | \widehat{\bm{\theta}}) + 2 k  
\end{equation}
where $\widehat{\bm{\theta}}$ is the maximum likelihood estimate of $\bm{\theta}$, and $k$ is the number of free parameters in $\bm{\theta}$. Since the noise variance is unknown, the parameter vector of the model, denoted by $\bm{\theta}_{k}$, is given by 
\begin{equation}
\bm{\theta}_{k} = [\lambda_{1},\ldots,\lambda_{k},\sigma^{2}]^{T}.
\end{equation}
There are thus $k+1$ free parameters in $\bm{\theta}_{k}$. Assuming that there are $k<\min(n,m)$ signals, the maximum likelihood estimate of the noise variance is given by \cite{Anderson63} (which Proposition \ref{prop:mandp fluctuations} corroborates in the $m<n$ setting)
\begin{equation}
\widehat{\sigma}^{2}_{(k)} = \dfrac{1}{n-k} \sum_{i=k+1}^{n} l_{i}
\end{equation}
where $l_{1} \geq \ldots \geq l_{n}$ are the eigenvalues of $\widehat{{\bf R}}$. Consider the test statistic
\begin{align}\label{eq:q aic}
q_{k} &= n \left[\dfrac{\frac{1}{n-k}\sum_{i=k+1}^{n} l_{i}^{2}}{\left(\widehat{\sigma}^{2}_{(k)}\right)^{2}}-(1+c)\right] - \left(\dfrac{2}{\beta}-1\right)c\\
&= n\left[\dfrac{\frac{1}{n-k}\sum_{i=k+1}^{n} l_{i}^{2}}{\left(\frac{1}{n-k} \sum_{i=k+1}^{n}l_{i}\right)^{2}}-(1+c)\right]-\left(\dfrac{2}{\beta}-1\right)c.
\end{align}
for a constant $c>0$. When $k>0$ signals are present and assuming $k \ll n$, then the distributional properties of the $n-k$ ``noise'' eigenvalues are closely approximated by the distributional properties of the eigenvalues given by Proposition \ref{prop:mandp fluctuations} of the signal-free SCM, {\it i.e.}, when $k=0$. It is hence reasonable to approximate the distribution of the statistic $q_{k}$ with the normal distribution whose mean and variance, for some $c>0$, given in Proposition \ref{prop:main prop}. The log-likelihood function $\log f(q_{k}| \widehat{{\bf \theta}})$, for large $n,m$ can hence be approximated by
\begin{equation}\label{eq:loglike approx}
-\log f(q_{k}| \widehat{{\bf \theta}}) \approx \dfrac{q_{k}^{2}}{2\frac{4}{\beta}c^{2}}+\underbrace{\dfrac{1}{2} \log 2\pi \frac{4}{\beta}c^{2}}_{{\rm Constant}}.
\end{equation}
In (\ref{eq:q aic}), and (\ref{eq:loglike approx}), it is reasonable (Bai and Silverstein provide an argument in \cite{BaiS04}) to use $c_{m} = n/m$ for the (unknown) limiting parameter $c = \lim n/m$. Plugging in $c \approx c_{m} = n/m$ into (\ref{eq:q aic}), and (\ref{eq:loglike approx}), ignoring the constant term on the right hand side of (\ref{eq:loglike approx}) when the log-likelihood function is substituted into (\ref{eq:aic criterion}) yields the estimator in (\ref{eq:new estimator}). Figure \ref{fig:score} plots sample realizations of the score function.

\begin{figure}[htp]
\centering
\subfigure[Complex signals: $n = 16$, $m = 32$.]{
\includegraphics[width=2.8in]{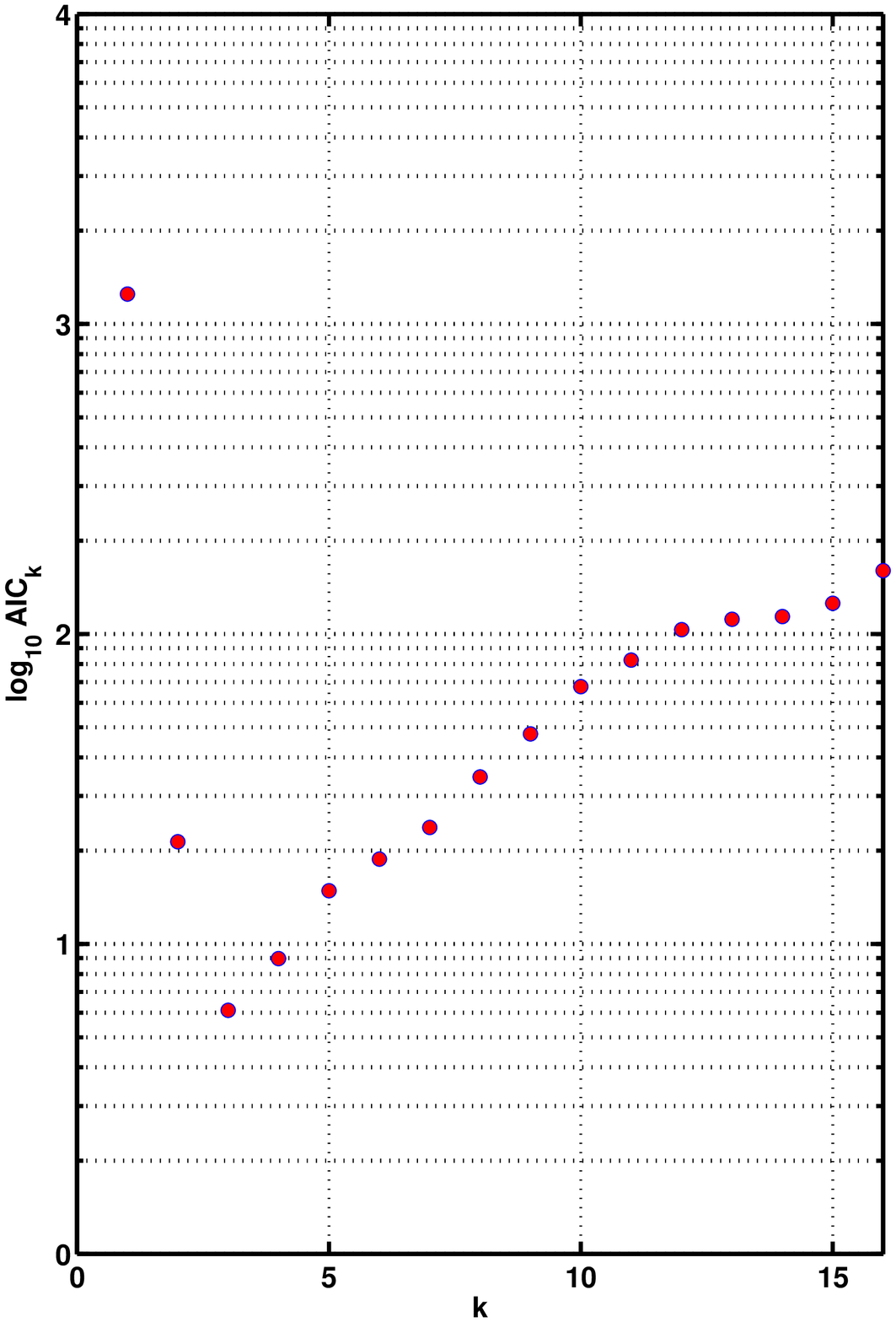}
}
\subfigure[Complex signals: $n = 32$, $m = 64$.]{
\includegraphics[width=2.8in]{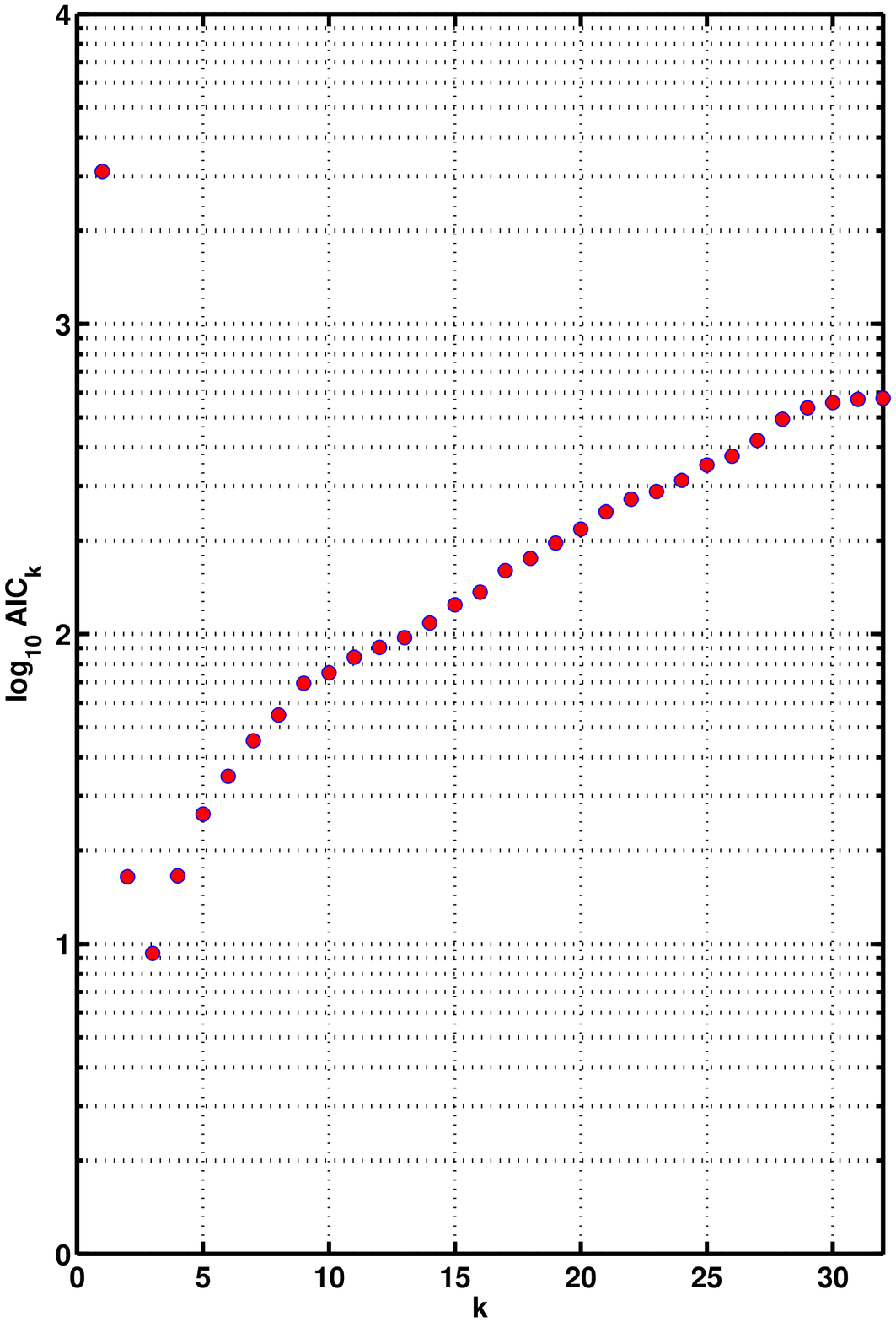}
}
\caption{Sample realizations of the proposed criterion when there $k=2$ complex valued signals and $\lambda_{1} = 10$, $\lambda_{2} = 3$ and $\lambda_{3} = \ldots = \lambda_{n} = 1$.}
\label{fig:score}
\end{figure}

\section{Extension to frequency domain and vector sensors}\label{sec:lrcf frequency domain}

When the $m$  snapshot vectors ${\bf x}_{i}(w_{j})$ for $j=1,\ldots,m$ represent Fourier coefficients vectors at frequency $w_{j}$  then the sample covariance matrix 
\begin{equation}
\widehat{{\bf R}}(w_{j}) = \dfrac{1}{m} \sum_{i=1}^{m} {\bf x}_{i}(w_{j}) {\bf x}_{i}(w_{j})'
\end{equation}
is the periodogram estimate of the spectral density matrix at frequency $w_{j}$. The time-domain approach carries over to the frequency domain so that the estimator in (\ref{eq:new estimator}) remains applicable with $l_{i} \equiv l_i(w_{j})$ where $l_{1}(w_{j}) \geq l_{2}(w_{j}) \geq \ldots \geq l_{n}(w_{j})$ are the eigenvalues of $\widehat{{\bf R}}(w_{j})$.

When the signals are wideband and occupy $M$ frequency bins, denoted by  $w_{1},\ldots,w_{M}$,  then the information on the number of signals present is contained in all the bins. The assumption that the observation time is much larger than the correlation times of the signals (sometimes referred to as the SPLOT assumption - stationary process, long observation time) ensures that the Fourier coefficients corresponding to the different frequencies are statistically independent. 

Thus the AIC based criterion for detecting the number of wideband signals that occupy the frequency bands $w_{1},\ldots,w_{M}$ is obtained by summing the corresponding criterion in (\ref{eq:new estimator}) over the frequency range of interest:

\theorembox{
\begin{subequations}
\begin{equation}
t_{j,k} = \left[(n-k)\dfrac{\sum_{i=k+1}^{n} l_{i}(w_{j})^{2}}{(\sum_{i=k+1}^{n} l_{i}(w_{j}))^{2}}- \left(1+\dfrac{n}{m}\right) \right]n-\left(\dfrac{2}{\beta}-1\right)\dfrac{n}{m}
\end{equation}
\begin{equation}
\hat{k}_{{\rm NEW}} = \argmin_{k\in \mathbb{N}:0\leq k < \min(n,m)} \sum_{j=1}^{M}\dfrac{\beta}{4} \left[\dfrac{m}{n}\right]^{2}\, t_{j,k}^{2} + 2M (k+1) 
\end{equation}
\end{subequations}
}

When the number of snapshots is severely constrained, the SPLOT assumption is likely to be violated so that the Fourier coefficients corresponding to different frequencies will not be statistically independent. This will likely degrade the performance of the proposed estimators.

When the measurement vectors represent quaternion valued narrowband signals, then $\beta =4$ so that the estimator in (\ref{eq:new estimator}) can be used. Quaternion valued vectors arise when the data collected from vector sensors is represented using quaternions as in \cite{miron06}.

\section{Consistency of the estimator and the effective number of identifiable signals}\label{sec:consistent number of signals}
For a fixed sample size, and system dimensionality, the probability of detecting a signal is the most practically useful criterion for comparing the performance of different estimators. For theoretical purposes, however, the large sample consistency of the estimator is (usually) more analytically tractable and hence often supplied as the justification for employing an estimator. We conjecture that the proposed algorithm is a consistent estimator of the true number of signals in the ``classical'' large sample asymptotic regime in the sense made explicit next.\\

\theorembox{
\begin{conjecture}\label{conj:classical consistency}
Let ${\bf R}$ be a $n \times n$ covariance matrix that satisfies the hypothesis of Proposition \ref{prop:spiked convergence}. Let $\widehat{{\bf R}}$ be a sample covariance matrix formed from $m$ snapshots. Then in the $n$ fixed, $m \to \infty$ limit, $\widehat{k}$ is a  consistent estimator of $k$ where $\widehat{k}$ is the estimate of the number of signals obtained using (\ref{eq:new estimator}).\\
\end{conjecture}
}

\noindent The ``classical'' notion of large sample consistency does not adequately capture the suitability of an estimator in  high dimensional, sample starved settings when $m < n$ or $m = O(n)$. In such settings, it is more natural to investigate the consistency properties of the estimator in the large system, large sample limit instead. We can use Proposition \ref{prop:spiked convergence} to establish an important property of the proposed estimator in such a limit.

\begin{theorem}\label{th:equivalent Rs}
Let ${\bf R}$ and $\widetilde{{\bf R}}$ be two $n\times n$ sized covariance matrices whose eigenvalues are related as
\begin{subequations}
\begin{equation}
{\bf \Lambda} = \textrm{diag}(\lambda_{1},\ldots,\lambda_{p},\lambda_{p+1},\ldots,\lambda_{k},\lambda,\ldots,\lambda)
\end{equation}
\begin{equation}
\widetilde{{\bf \Lambda}} = \textrm{diag}(\lambda_{1},\ldots,\lambda_{p},\lambda,\ldots,\lambda)
\end{equation}
\end{subequations}
where for some $c \in (0,\infty)$, and all $i = p+1,\ldots,k$, $\lambda < \lambda_{i} \leq \lambda\, (1+\sqrt{c})$. Let $\widehat{{\bf R}}$ and $\widehat{\widetilde{{\bf R}}}$ be the associated sample covariance matrices formed from $m$ snapshots. Then for every $n,m(n) \to \infty$ such that $c_{m}= n/m \to c$, 
\begin{subequations}\label{eq:equivalent Rs}
\begin{equation}
\textrm{Prob}(\widehat{k} = j \, | \, {\bf R}) \to \textrm{Prob}(\widehat{k} = j \, | \, \widetilde{{\bf R}})  \qquad \textrm{for } j = 1,\ldots, p
\end{equation}
\textrm{and}
\begin{equation}
  \textrm{Prob}(\widehat{k} > p \, | \, {\bf R}) \to  \textrm{Prob}(\widehat{k} > p \, | \, \widetilde{{\bf R}})
\end{equation}
\end{subequations}
where the convergence is almost surely and $\widehat{k}$ is the estimate of the number of signals obtained using the algorithm in (\ref{eq:new estimator}).
\end{theorem}
\begin{proof}
The result follows from Proposition \ref{prop:spiked convergence}. The almost sure convergence of the sample eigenvalues $l_{j} \to \lambda(1+\sqrt{c})^{2}$ for $j = p+1,\ldots,k$ implies that $i$-th largest eigenvalues of $\widehat{{\bf R}}$ and $\widehat{\widetilde{{\bf R}}}$, for $i=1,\ldots,p+1$, converge to the same limit almost surely. The fluctuations about this limit will hence be identical so that (\ref{eq:equivalent Rs}) follows in the asymptotic limit. 
\end{proof}

Note that the rate of convergence to the asymptotic limit for $\textrm{Prob}(\widehat{k} > p \, | \, {\bf R})$ and $\textrm{Prob}(\widehat{k} > p \, | \, \widetilde{{\bf R}})$ will, in general, depend on the eigenvalue structure of ${\bf R}$ and may be arbitrarily slow. Thus, Theorem \ref{th:equivalent Rs} yields \underline{no insight into rate of convergence type issues} which are important in practice. Rather, the theorem is a statement on the asymptotic equivalence, from an identifiability point of view, of sequences of sample covariance matrices which are related in the manner described. At this point, we are unable to prove the consistency of the proposed estimator as this would require more a refined analysis that characterizes the fluctuations of subsets of the (ordered) ``noise'' eigenvalues. The statement regarding consistency of the proposed estimator, in the sense of large system, large sample limit, is presented as a conjecture with numerical simulations used as non-definitive yet corroborating evidence.\\

\theorembox{
\begin{conjecture}\label{conj:consistency}
Let ${\bf R}$ be a $n \times n$ covariance matrix that satisfies the hypothesis of Proposition \ref{prop:spiked convergence}. Let $\widehat{{\bf R}}$ be a sample covariance matrix formed from $m$ snapshots. Define
\begin{equation}\label{eq:effective rank}
k_{{\rm eff}}(c \, | \, {\bf R}) := \textrm{Number of eigenvalues of }{\bf R} > \lambda(1+\sqrt{c}).
\end{equation}
Then in $m,n \to \infty$ limit with $c_{m}=n/m \to c$, $\widehat{k}$ is a  consistent estimator of $k_{{\rm eff}}(c)$ where $\widehat{k}$ is the estimate of the number of signals obtained using the algorithm in (\ref{eq:new estimator}).\\
\end{conjecture}
}
\vspace{0.1in}
\noindent
Motivated by Proposition \ref{prop:spiked convergence}, we (heuristically) define the \textit{effective number of (identifiable) signals} as
\begin{equation}\label{eq:effective number of signals}
k_{{\rm eff}}({\bf R}) = \#\textrm{ eigs. of }{\bf R} > \sigma^{2}\left(1+\sqrt{\dfrac{n}{m}}\right).
\end{equation}
Conjecture \ref{conj:consistency} then simply states that the proposed estimator is a consistent estimator of the effective number of (identifiable) signals in the large system, large sample limit.

\subsection{The asymptotic identifiability of two closely spaced signals}

Suppose there are two uncorrelated (hence, independent) signals so that ${\bf R}_{s} = \textrm{diag}(\sigma_{{\rm S}1}^{2},\sigma_{{\rm S}2}^{2})$. In (\ref{eq:superposition problem}) let ${\bf A} = [{\bf v}_{1} {\bf v}_{2}]$. In a sensor array processing application, we think of ${\bf v}_{1} \equiv {\bf v}(\theta_{1})$ and ${\bf v}_{2} \equiv {\bf v}_{2}(\theta_{2})$ as encoding the array manifold vectors for a source and an interferer with powers $\sigma_{{\rm S}1}^{2}$ and $\sigma_{{\rm S}2}^{2}$, located at $\theta_{1}$ and $\theta_{2}$, respectively. The covariance matrix given by 
\begin{equation}
{\bf R} = \sigma_{{\rm S}1}^{2} {\bf v}_{1}{\bf v}_{1}'+ \sigma_{{\rm S}2}^{2}  {\bf v}_{2}{\bf v}_{2}' + \sigma^{2} {\bf I}
\end{equation}
has the $n-2$ smallest eigenvalues $\lambda_{3} = \ldots = \lambda_{n} = \sigma^{2}$ and the two largest eigenvalues
\begin{subequations}\label{eq:ev 2 sources}
\begin{equation}
\lambda_{1} =  
\sigma^{2}+ \dfrac{\left(\sigma_{{\rm S}1}^{2} \parallel \! {\bf v}_{1} \!\parallel^{2}+\sigma_{{\rm S}2}^{2} \parallel \! {\bf v}_{2} \!\parallel^{2}\right)}{2} + \dfrac{
\sqrt{\left(\sigma_{{\rm S}1}^{2} \parallel \! {\bf v}_{1} \!\parallel^{2}-\sigma_{{\rm S}2}^{2} \parallel \! {\bf v}_{2} \!\parallel^{2}\right)^{2}+4\sigma_{{\rm S}1}^{2}\sigma_{{\rm S}2}^{2} |\langle {\bf v}_{1}, {\bf v}_{2} \rangle| ^{2}}}{2}
\end{equation}
\begin{equation}
\lambda_{2} =  \sigma^{2}+ \dfrac{\left(\sigma_{{\rm S}1}^{2} \parallel \! {\bf v}_{1} \!\parallel^{2}+\sigma_{{\rm S}2}^{2} \parallel \! {\bf v}_{2} \!\parallel^{2}\right)}{2} -\dfrac{
\sqrt{\left(\sigma_{{\rm S}1}^{2} \parallel \! {\bf v}_{1} \!\parallel^{2}-\sigma_{{\rm S}2}^{2} \parallel \! {\bf v}_{2} \!\parallel^{2}\right)^{2}+4\sigma_{{\rm S}1}^{2}\sigma_{{\rm S}2}^{2} |\langle {\bf v}_{1}, {\bf v}_{2} \rangle| ^{2}}}{2}
\end{equation}
\end{subequations}
respectively. Applying the result in Proposition \ref{prop:spiked convergence} allows us to express the effective number of signals as  
\begin{equation}
k_{{\rm eff}} = 
\begin{cases}
2 &\qquad\textrm{if    } \phantom{~~~~}\sigma^{2} \left(1+\sqrt{\dfrac{n}{m}} \right) < \lambda_{2}\\
  &\\
1 &\qquad\textrm{if    } \phantom{~~~~}\lambda_{2} \leq \sigma^{2} \left(1+\sqrt{\dfrac{n}{m}} \right) < \lambda_{1}\\
   &\\
0 &\qquad\textrm{if    } \phantom{~~~~}\lambda_{1} \leq \sigma^{2} \left(1+\sqrt{\dfrac{n}{m}} \right) \\
\end{cases}
\end{equation}
In the special situation when  $\parallel\! {\bf v}_{1} \!\parallel = \parallel\! {\bf v}_{2}\! \parallel = \parallel\! {\bf v}\! \parallel$ and $\sigma_{{\rm S1}}^{2} = \sigma_{{\rm S2}}^{2} = \sigma_{{\rm S}}^{2}$, we can (in an asymptotic sense) reliably detect the presence of \textit{both signals} from the sample eigenvalues alone whenever
\begin{empheq}[
box=\setlength{\fboxrule}{1pt}\fbox]{equation}
\label{eq:array tradeoff}
\textrm{Asymptotic identifiability condition}: \qquad \sigma_{{\rm S}}^{2} \parallel \! {\bf v} \!\parallel^{2} \left(1-\dfrac{|\langle {\bf v}_{1},{\bf v}_{2}\rangle |}{\parallel \! {\bf v}\parallel}\right)  > \sigma^{2} \sqrt{\dfrac{n}{m}}  
\end{empheq}
Equation (\ref{eq:array tradeoff}) captures the tradeoff between the identifiability of two closely spaced signals, the dimensionality of the system, the number of available snapshots and the cosine of the angle between the vectors ${\bf v}_{1}$ and ${\bf v}_{2}$.

We note that the concept of the effective number of signals is an asymptotic concept for large dimensions and relatively large sample sizes. For moderate dimensions and sample sizes, the fluctuations in the ``signal'' and ``noise'' eigenvalues affect the reliability of the underlying detection procedure as illustrated in Figure \ref{fig:signal blurs}. From Proposition \ref{prop:spiked convergence}, we expect that the largest ``noise'' eigenvalue will, with high probability, be found in a neighborhood around $\sigma^{2}(1+\sqrt{n/m})^{2}$ while the ``signal'' eigenvalues will, with high probability, be found in a neighborhood around $\lambda_{j} \left( 1+ \frac{\sigma^{2}\,n}{m(\lambda_{j}-\sigma^{2})}\right)$. From Proposition \ref{prop:spiked fluctuations}, we expect the ``signal'' eigenvalues to exhibit Gaussian fluctuations with a standard deviation of approximately $\sqrt{\frac{2}{\beta n}\lambda_{j}^{2} \left(1- \frac{n}{m(\lambda_{j}-\sigma^{2})^{2}} \right)}$. This motivates our definition of the metric $Z^{{\rm Sep}}_{j}$ given by
\begin{empheq}[
box=\setlength{\fboxrule}{1pt}\fbox]{equation}\label{eq:zjsep}
Z^{{\rm Sep}}_{j} := \dfrac{\lambda_{j} \left( 1+ \dfrac{\sigma^{2}\,n}{m(\lambda_{j}-\sigma^{2})}\right) - \sigma^{2}\left(1+\sqrt{\dfrac{n}{m}}\right)^{2}}{\sqrt{\dfrac{2}{\beta n}\lambda_{j}^{2} \left(1- \dfrac{n}{m(\lambda_{j}-\sigma^{2})^{2}} \right)}},
\end{empheq}
then measures the (theoretical) separation of the $j$-th ``signal'' eigenvalue from the largest ``noise'' eigenvalue in standard deviations of the $j$-the signal eigenvalue's fluctuations. Simulations suggest that reliable detection (with an empirical probability greater than 90\%) of the effective number of signals is possible if $Z^{{\rm Sep}}_{j}$ is larger than $5-15$. This large range of values for the minimum $Z^{{\rm Sep}}_{j}$, which we obtained from the results in Section \ref{sec:lrcf simulations}, suggests that a more precise characterization of the finite system, finite sample performance of the estimator will have to take into account the more complicated-to-analyze interactions between the ``noise'' and the ``signal'' eigenvalues that are negligible in the large system, large sample limit. Nonetheless, because of the nature of the random matrix results on which our guidance is based, we expect our heuristics to be more accurate in high-dimensional, relatively large sample size settings than the those proposed in \cite{kaveh87a} and \cite{fishler00a} which rely on Anderson's classical large sample asymptotics.

\begin{figure}[t]
\centering
\subfigure[When $n$ and $m$ large enough, so that the (largest) ``signal'' eigenvalue is sufficiently separated from the ``noise'' eigenvalues, then reliable detection is possible.]{
\label{fig:signal separates}
\psfrag{nminus2}{\makebox(0,0){$3$}}
\psfrag{nminus1}{\makebox(0,0){$2$}}
\psfrag{n}{\makebox(0,0){$1$}}
\psfrag{wiggling}{\makebox(0,0){\wiggling}}
\psfrag{wiggling2}{\makebox(0,0){\wigglingtwo}}
\includegraphics*[width=0.75\textwidth]{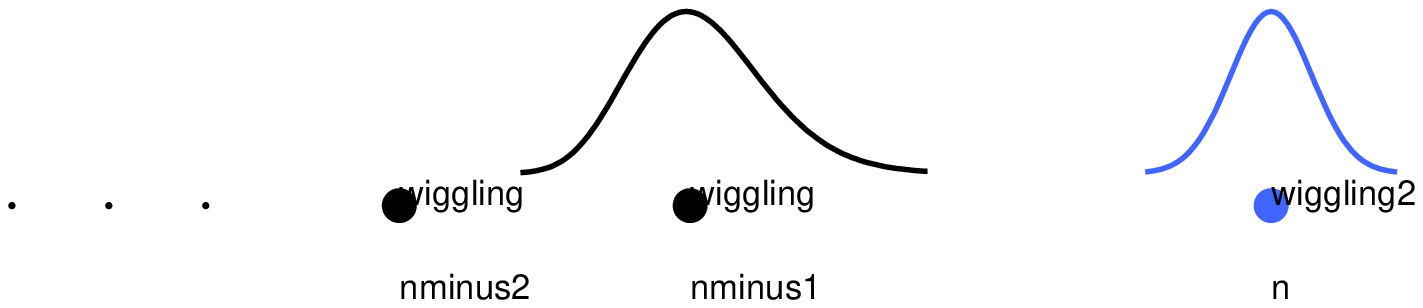}
}\\
\subfigure[When $n$ and $m$ are not quite large enough so that the (largest) ``signal'' eigenvalue is \textit{not} sufficiently separated from the largest ``noise'' eigenvalue, then reliable detection becomes challenging.]{
\label{fig:signal blurs}
\psfrag{nminus2}{\makebox(0,0){$3$}}
\psfrag{nminus1}{\makebox(0,0){$2$}}
\psfrag{n}{\makebox(0,0){$1$}}
\psfrag{wiggling}{\makebox(0,0){\wiggling}}
\psfrag{wiggling2}{\makebox(0,0){\wigglingtwo}}
\psfrag{wiggling3}{\makebox(0,0){\wigglingthree}}
\psfrag{wiggling4}{\makebox(0,0){\wigglingfour}}
\includegraphics*[width=.7\textwidth]{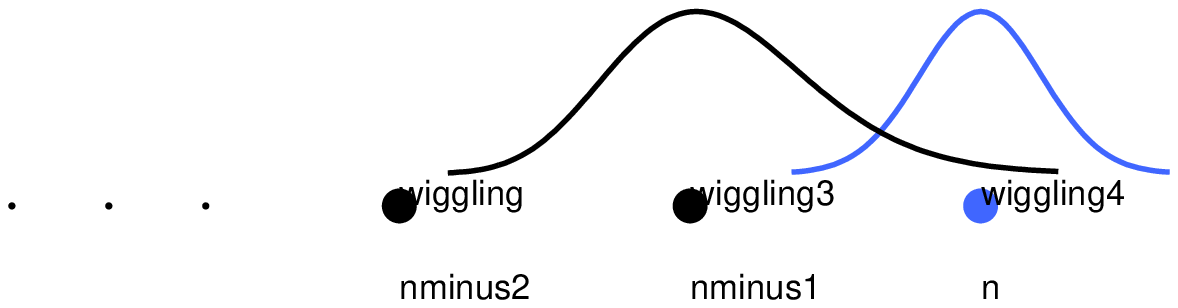}
}
\caption{The finite system dimensionality and sample size induced fluctuations of ``signal'' (blue) and ``noise'' (black) eigenvalues about their limiting positions are shown. The magnitude of the fluctuations impacts the ability to discriminate the ``signal'' eigenvalue from the largest ``noise'' eigenvalue.}
\label{tab:signal and noise fluctuations} 
\end{figure}

\section{Numerical simulations}\label{sec:lrcf simulations}

We now illustrate the performance of our estimator using Monte-Carlo simulations. The results obtained provide evidence for the consistency properties conjectured in Section \ref{sec:consistent number of signals}. In all of our simulations, we use a population covariance matrix ${\bf R}$ that has arbitrarily fixed, yet unknown, eigenvectors, $k=2$ ``signal'' eigenvalues with $\lambda_{1}=10$ and $\lambda_{2}=3$, and  $n-2$ ``noise'' eigenvalues with $\lambda_{3}=\ldots=\lambda_{n}=\lambda = \sigma^{2}=1$. We assume that the snapshot vectors ${\bf x}_{i}$ modelled as in (\ref{eq:superposition problem}) are complex valued so that we must plug in $\beta =2$ in (\ref{eq:new estimator}); the choice of complex valued signals is motivated by our focus on array signal processing/wireless communications applications. 

Over $4000$ Monte-Carlo simulations, and various $n$ and $m$, we  obtain an estimate of the number of signals from the eigenvalues of the sample covariance matrix using our new estimator and the modified Wax-Kailath estimator, described in (\ref{eq:new estimator}) and (\ref{eq:WK MDL mod}) respectively. We do not consider the Wax-Kailath AIC estimator in (\ref{eq:aic est}) in our simulations because of its proven \cite{kailath-wax} inconsistency in the fixed system size, large sample limit -  we are interested in estimators that exhibit the consistency conjectured in Section \ref{sec:consistent number of signals} in \textit{both} asymptotic regimes. A thorough comparison of the performance of our estimator with other estimators (and their ad-hoc modifications) found in the literature is beyond the scope of this article. 

\begin{figure}
\centering
\subfigure[Empirical probability that $\widehat{k} =2$ for various $n$ and $m$.]{
\label{fig:true consistency}
\includegraphics[width=6.05in]{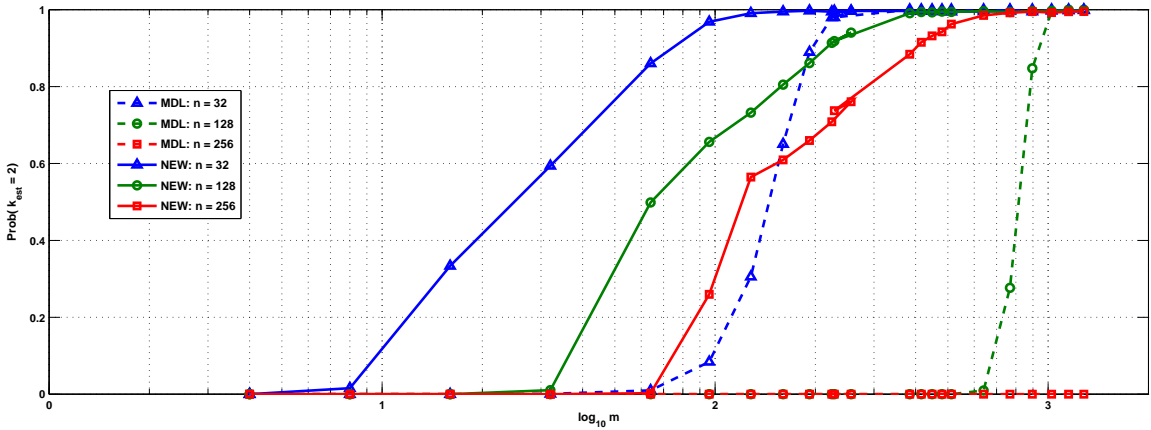}
}\\
\subfigure[Empirical probability that $\widehat{k} =1$ for various $n$ and $m$.]{
\label{fig:sim1 phat1}
\includegraphics[width=6.05in]{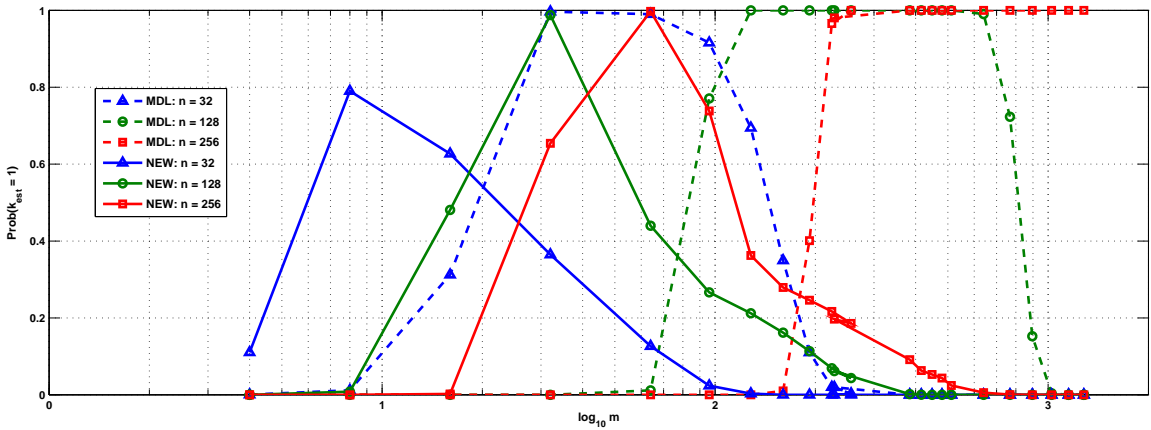}
}\\
\subfigure[Empirical probability that $\widehat{k} =0$ for various $n$ and $m$.]{
\label{fig:sim1 phat0}
\includegraphics[width=6.05in]{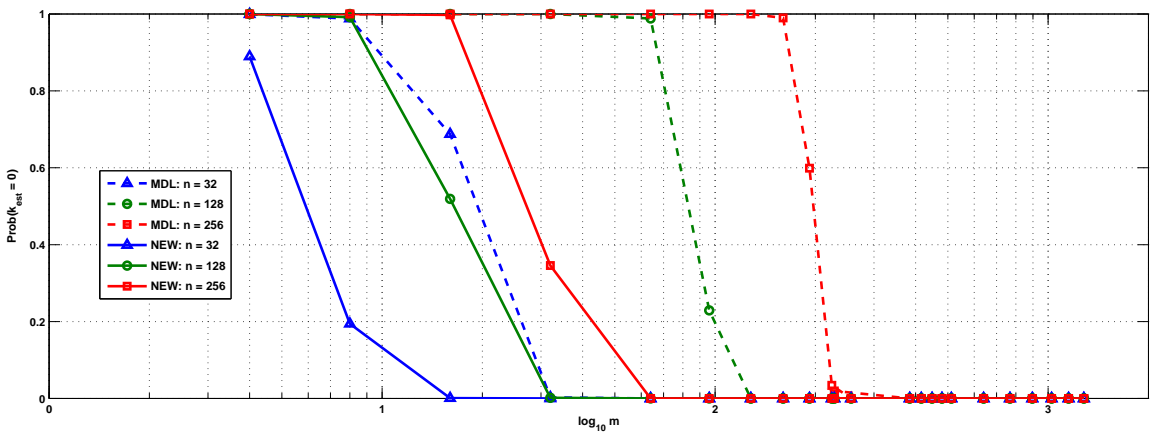}
}\\
\caption{Comparison of the performance of the new estimator in (\ref{eq:new estimator}) with the MDL estimator in (\ref{eq:WK MDL mod}) for various $n$ (system size) and $m$ (sample size).}
\label{fig:sim1}
\end{figure}

We first investigate the large sample consistency in the classical sense of $n$ fixed and $m \to \infty$. For a choice of $n$, and different values of $m$ we compute the empirical probability of detecting two signals. For large values of $m$ we expect both the new and the Wax-Kailath MDL estimator to detect both signals with high probability. Figure \ref{fig:sim1} plots the results obtained in the numerical simulations.

Figure \ref{fig:true consistency} shows that for $n = 32,128$, if $m$ is large enough then either estimator is able to detect both signals with high probability. However, the new estimator requires significantly less samples to do so than the Wax-Kailath MDL estimator. 

\begin{figure}[h]
\centering
\includegraphics[width=6.5in]{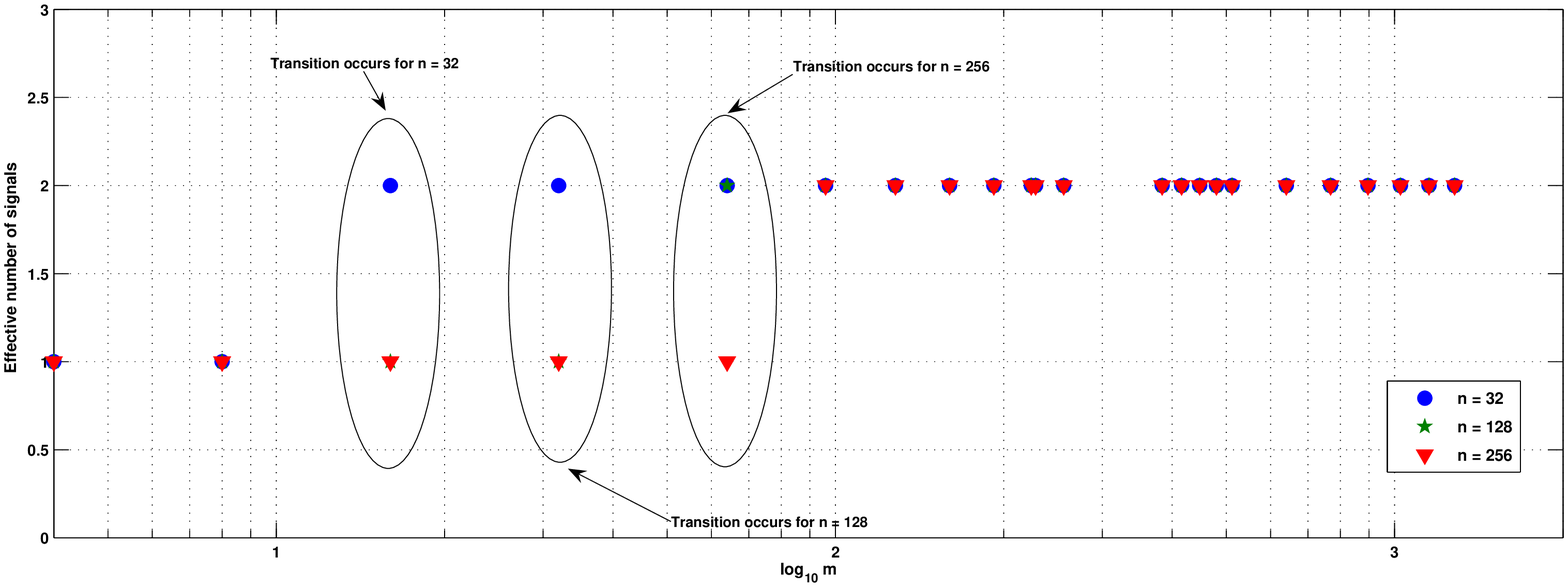}
\caption{The effective number of identifiable signals, computed using (\ref{eq:effective number of signals}) for the values of $n$ (system size) and $m$ (sample size) considered in Figure \ref{fig:sim1} when the population covariance matrix has two signal eigenvalues $\lambda_{1} =10$ and $\lambda_{2}=3$ and $n-2$ noise eigenvalues $\lambda_{3}=\ldots=\lambda_{n}=\sigma^{2}=1$.}
\label{fig:expn for simunder}
\end{figure}

Figures \ref{fig:sim1 phat1} and \ref{fig:sim1 phat0} plot the empirical probability of detecting one and zero signals, respectively, as a function of $m$ for various values of $n$. The results exhibit the chronically reported symptom of estimators underestimating the number of signals - this is not surprising given the discussion in Section \ref{sec:consistent number of signals}. Figure \ref{fig:expn for simunder} plots the effective number of identifiable signals $k_{eff}$, determined using (\ref{eq:effective number of signals}) for the various values of $n$ and $m$ considered. We observe that the values of $n$ and $m$ for which the empirical probability of the new estimator detecting one signal is high also correspond to regimes where $k_{eff} =1$. This suggests that the asymptotic concept of the effective number of signals remains relevant in a non-asymptotic regime as well. At the same time, however, one should not expect the signal identifiability/unidentifiability predictions in Section \ref{sec:consistent number of signals} to be accurate in the severely sample starved settings where $m \ll n$. For example, Figure \ref{fig:sim1 phat0} reveals that the new estimator detects zero signals with high empirical probability when there are less than $10$ samples available even though $k_{eff}=1$ in this regime from Figure \ref{fig:expn for simunder}. In the large system, relatively large sample size asymptotic limit, however, these predictions are accurate - we discuss this next.

When $m = 4n$ samples are available, Figure \ref{fig:numerics 1} shows that the proposed estimator consistently detects two signals while the Wax-Kailath MDL estimator does not. However, when $m = n/4$ samples are available, Figure \ref{fig:numerics 1} suggests that neither estimator is able to detect both the signals present. A closer examination of the empirical data presents a different picture. The population covariance has two signal eigenvalues $\lambda_{1}=10$ and $\lambda_{2}=3$ with the noise eigenvalues $\sigma^{2}=1$. Hence, when $m = n/4$, from (\ref{eq:effective rank}), the effective number of signals $k_{eff}=1$. Figure \ref{fig:numerics 2} shows that for large $n$ and  $m = n/4$, the new estimator consistently estimates one signal, as expected. We remark that that the signal eigenvalue $\lambda_{2}$ which is asymptotically unidentifiable falls exactly on the threshold in (\ref{eq:effective rank}). The consistency of the new estimator with respect to the effective number of signals corroborates the asymptotic tightness of the fundamental limit of sample eigenvalue based detection. On inspecting Tables \ref{tab:sim2}-(b) and \ref{tab:sim2}-(d) it is evident that the Wax-Kailath MDL estimator consistently underestimates the effective number of signals in the large system, large sample size limit.

Table \ref{tab:sim3} provides additional evidence for Conjecture \ref{conj:consistency}. We offer Tables \ref{tab:sim3}-(c) and \ref{tab:sim3}-(d) as evidence for the observation that large system, large sample consistency aside, the rate of convergence can be arbitrary slow and cannot be entirely explained by the metric $\z{k_{eff}}$ in (\ref{eq:zjsep}).

\begin{figure}[t]
\centering
\subfigure[Empirical probability that $\widehat{k} =2$ for various $n$ and fixed $n/m$.]
{
\label{fig:numerics 1}
\includegraphics[width=5.2in]{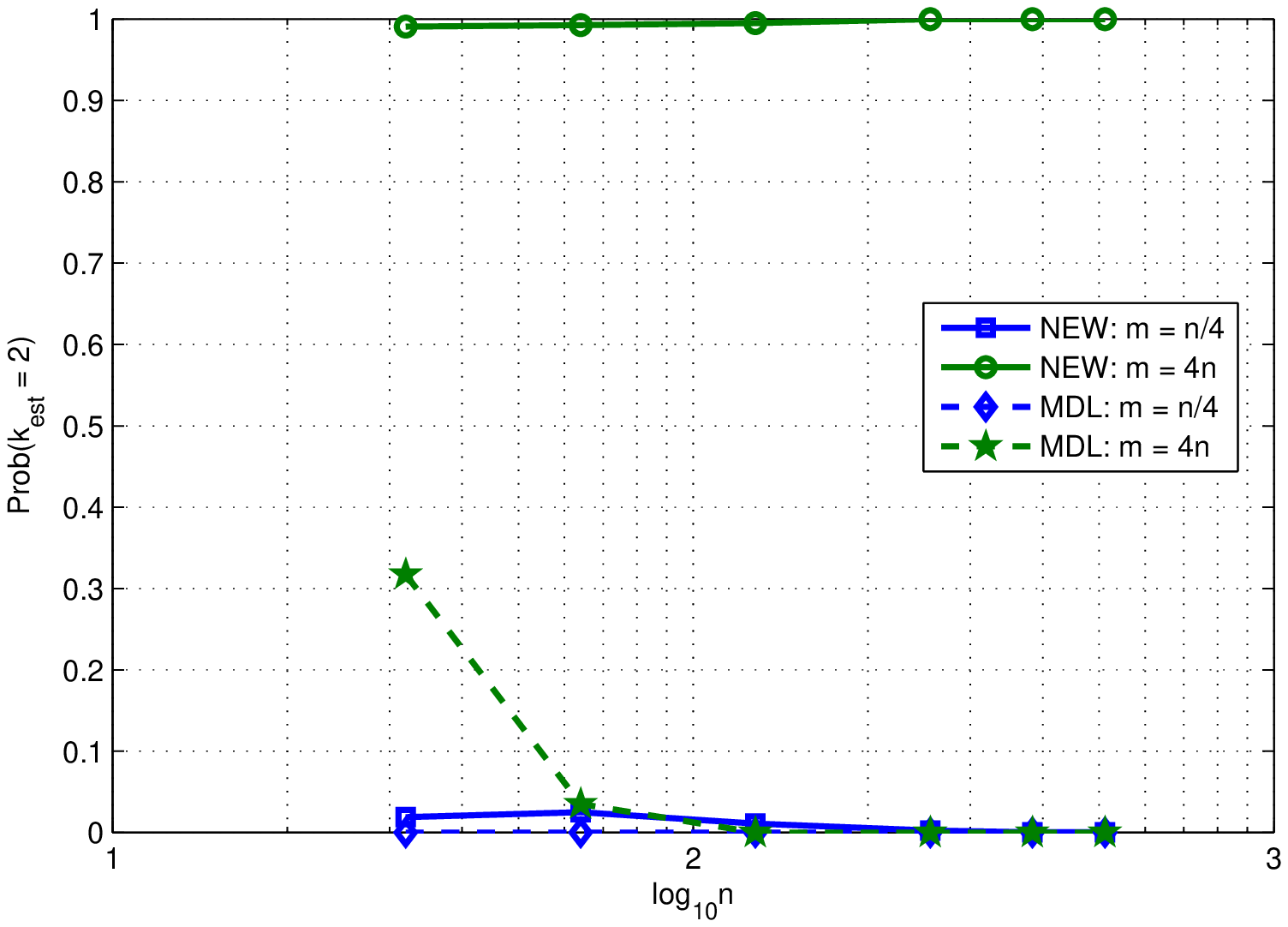}
}\\
\subfigure[Empirical probability that $\widehat{k} =1$ for various $n$ and fixed $n/m$.]
{
\label{fig:numerics 2}
\includegraphics[width=5.2in]{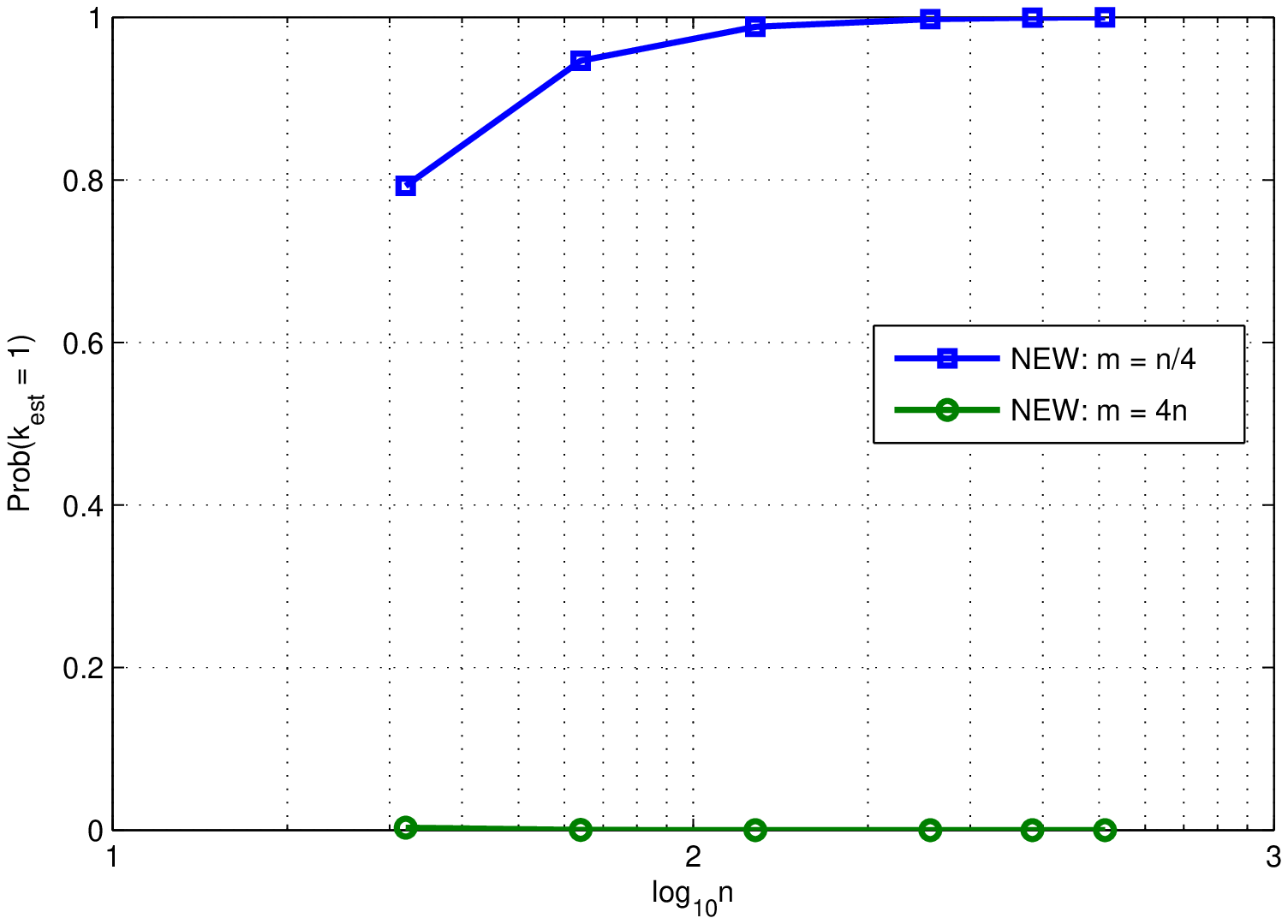}
}
\caption{Comparison of the performance of the new estimator in (\ref{eq:new estimator}) with the MDL estimator in (\ref{eq:WK MDL mod}) for various values of $n$ (system size) with $m$ (sample size) such that $n/m$ is fixed.}
\label{fig:numerics}
\end{figure}

\begin{table}
\centering
\includegraphics[width=5.5in]{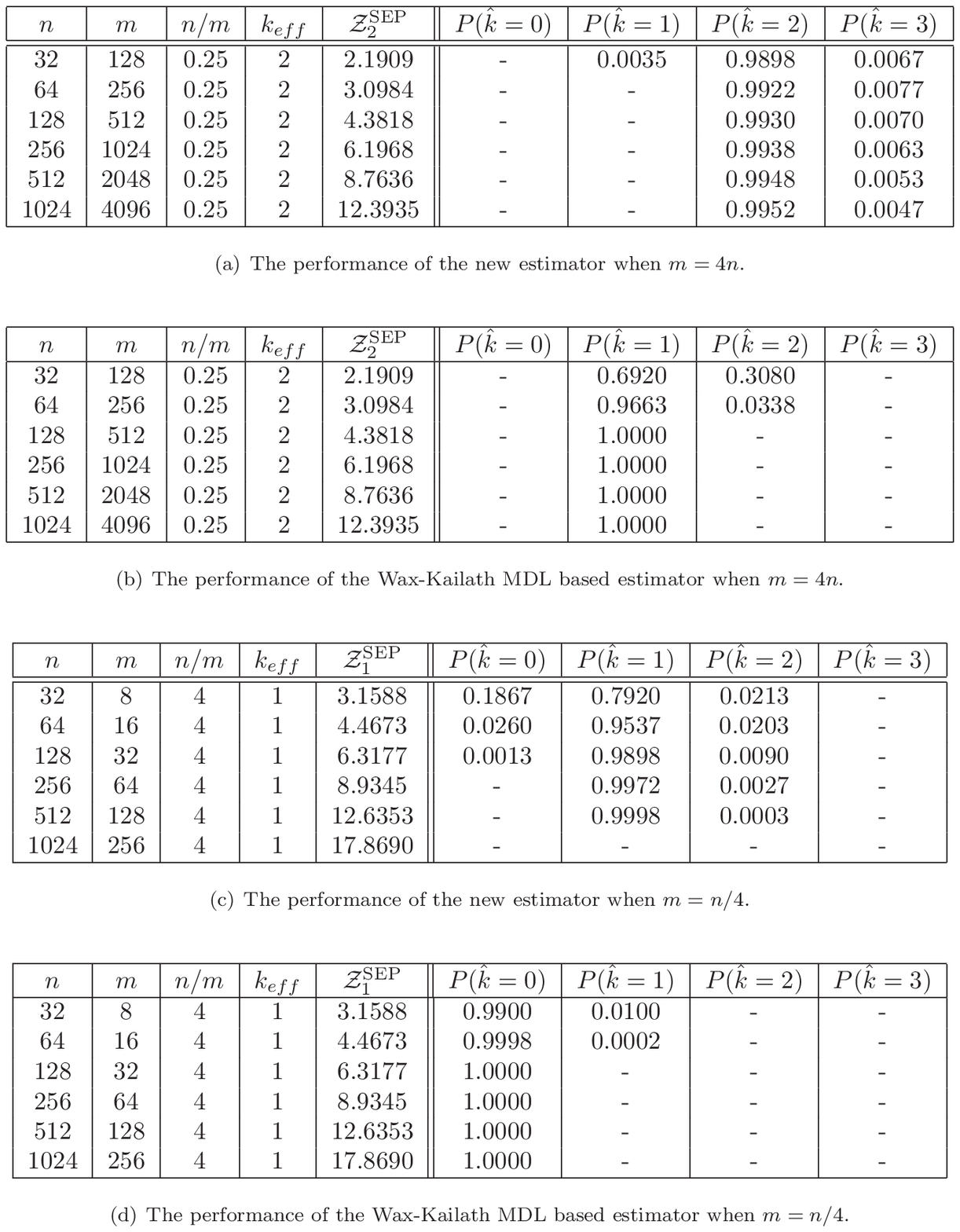}
\caption{Comparison of the empirical performance of the new estimator in (\ref{eq:new estimator}) with the Wax-Kailath MDL estimator in (\ref{eq:WK MDL mod}) when the population covariance matrix has two signal eigenvalues $\lambda_{1} =10$ and $\lambda_{2}=3$ and $n-2$ noise eigenvalues $\lambda_{3}=\ldots=\lambda_{n}=\sigma^{2}=1$.  The effective number of identifiable signals is computed using (\ref{eq:effective number of signals}) while the separation metric $\z{j}$ is computed for $j=k_{eff}$ using (\ref{eq:zjsep}). Here $n$ denotes the system size, $m$ denotes the sample size and the snapshot vectors, modelled in (\ref{eq:superposition problem}) are taken to be complex-valued.}
\label{tab:sim2}
\end{table}

\begin{table}
\centering
\includegraphics[width=5.5in]{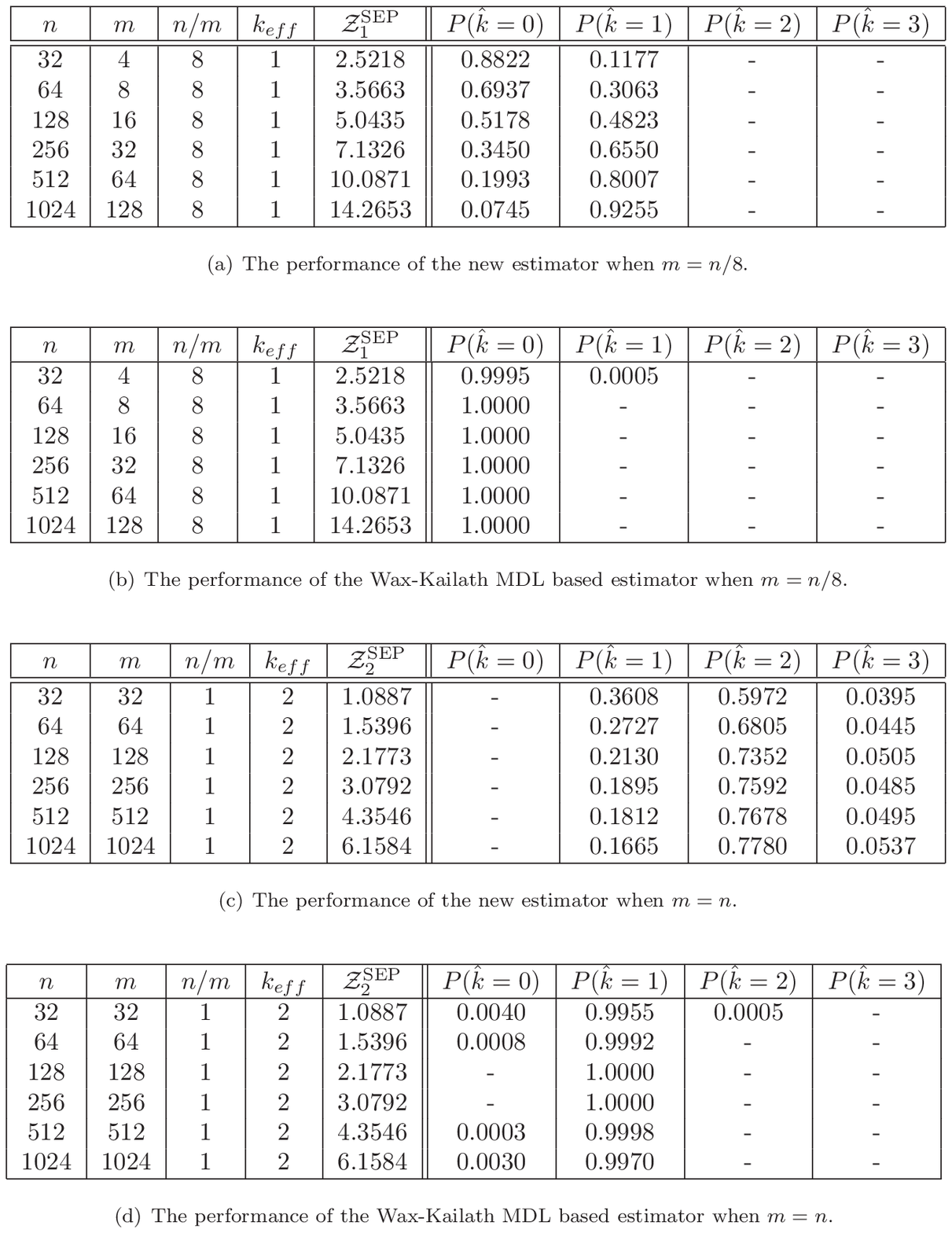}
\caption{Comparison of the empirical performance of the new estimator in (\ref{eq:new estimator}) with the Wax-Kailath MDL estimator in (\ref{eq:WK MDL mod}) when the population covariance matrix has two signal eigenvalues $\lambda_{1} =10$ and $\lambda_{2}=3$ and $n-2$ noise eigenvalues $\lambda_{3}=\ldots=\lambda_{n}=\sigma^{2}=1$.  The effective number of identifiable signals is computed using (\ref{eq:effective number of signals}) while the separation metric $\z{j}$ is computed for $j=k_{eff}$ using (\ref{eq:zjsep}). Here $n$ denotes the system size, $m$ denotes the sample size and the snapshot vectors, modelled in (\ref{eq:superposition problem}) are taken to be complex-valued.}
\label{tab:sim3}
\end{table}

\section{Concluding remarks}\label{sec:lrcf conclusion}

We have developed an information theoretic approach for detecting the number of signals in white noise from the sample eigenvalues alone. The proposed estimator explicitly takes into account the blurring of the sample eigenvalues due to the finite size. The stated conjecture on the consistency of the algorithm, in both the $n$ fixed, $m \to \infty$ sense and the $n,m(n) \to \infty$ with $n/m(n) \to c$ sense remains to be proven. It would be interesting to investigate the impact of a broader class of penalty functions on the consistency, strong or otherwise, in both asymptotic regimes, in the spirit of \cite{zhao86a}. 

In future work, we plan to address the problem of estimating the number of high-dimensional signals in noise with arbitrary covariance \cite{zhao86b}, using relatively few samples when an independent estimate of the noise sample covariance matrix, that is itself formed from relative few samples, is available. This estimator will also be of the form in \ref{eq:new estimator} and will exploit the analytical characterization of properties of the traces of powers of random Wishart matrices with a covariance structure that is also random \cite{raj05:annals}.

It remains an open question to analyze such signal detection algorithms in the Neyman-Pearson sense of finding the most powerful test that does not exceed a threshold probability of false detection. Finer properties, perhaps buried in the rate of convergence to the asymptotic results used, might be useful in this context. In the spirit of Wax and Kailath's original work, we developed a procedure that did not require us to make any subjective decisions on setting threshold levels. Thus, we did not consider largest eigenvalue tests in sample starved settings of the sort developed in \cite{johnstone01,imj06a} and the references therein. Nevertheless, if the performance can be significantly improved using a sequence of nested hypothesis tests, then this might be a price we might be ready to pay. This is especially true for the detection of low-level signals right around the threshold where the asymptotic results suggest that it becomes increasingly difficult, if not impossible, to detect signals using the sample eigenvalues alone.

\section*{Acknowledgements}

We thank Arthur Baggeroer, William Ballance and the anonymous reviewers for their feedback and encouragement. We are especially grateful to the associate editor, Erik Larsson, for his accommodating our multiple requests for extensions so that we could incorporate the reviewers' excellent suggestions into the manuscript. The first author was supported by an Office of Naval Research Special Postdoctoral award in Ocean Acoustics under grant N00014-07-1-0269. The authors were partially supported by NSF Grant DMS-0411962.

%
%



\bibliographystyle{IEEEtran}
\bibliography{randbib}
%
\end{document}